\documentclass[12pt]{article}
\usepackage{amssymb,amsmath,amsfonts,amsthm,enumerate,color}
\usepackage[toc,page,title,titletoc,header]{appendix}
\usepackage{graphicx}
\usepackage{cite}
\usepackage{savesym}
\usepackage{mathrsfs}
\usepackage{subfigure}
\usepackage[latin1]{inputenc}
\usepackage[T1]{fontenc}
\usepackage{indentfirst}
\usepackage{eurosym}
\usepackage{multicol}
\usepackage{booktabs}
\usepackage{hyperref}
\usepackage[T1]{fontenc}

\numberwithin{equation}{section}

\newtheorem{Theorem}{Theorem}
\newtheorem{Lemma}{Lemma}
\newtheorem{Assumption}{Assumption}
\newtheorem{Proposition}{Proposition}
\newtheorem{Definition}{Definition}
\newtheorem{Corollary}{Corollary}

\newtheorem{Remark}{Remark}

\numberwithin{equation}{section}

\def\embed{\hookrightarrow}

\def\F{\mathcal F}

\def\la{\langle}
\def\ra{\rangle}

\def\xah{\hat{x}_t}
\def\xa{x_t}
\def\xaus{x_{\textup{aux}}}
\def\xausd{x_{\textup{aux}}(\delta)}
\def\xaush{\hat{x}_{\textup{aux}}}

\def\adp{\kappa_{\textup{DP}}}
\def\C{\mathbb{C}}
\def\xda{x^\delta_{\kappa_{\textup{DP}}}}

\def\Tda{T^\delta_\kappa}

\def\xdag{{x^\dag}}

\def\Lom3{L^2(\Omega)^3}

\def\yd{y^\delta}

 \def\p{\partial} 
\def \Vh0{\stackrel{\circ}{V}_h} \def\to{\rightarrow}

    \def\R{{\mathbb R}}
  \def\f{\frac}  

\def\p{\partial}

\def\E{{\mathcal E}}  

\newcommand{\lc}
{\mathrel{\raise2pt\hbox{${\mathop<\limits_{\raise1pt\hbox
{\mbox{$\sim$}}}}$}}}

\newcommand{\gc}
{\mathrel{\raise2pt\hbox{${\mathop>\limits_{\raise1pt\hbox{\mbox{$\sim$}}}}$}}}

\newcommand{\ec}
{\mathrel{\raise2pt\hbox{${\mathop=\limits_{\raise1pt\hbox{\mbox{$\sim$}}}}$}}}

\def\bb{\begin{equation}} \def\ee{\end{equation}}

\def\beqn{\begin{eqnarray}}  \def\eqn{\end{eqnarray}}
\def\beq{\begin{equation}} \def\eeq{\end{equation}}

\def\beqnx{\begin{eqnarray*}} \def\eqnx{\end{eqnarray*}}

\def\bn{\begin{enumerate}} \def\en{\end{enumerate}}

\def\bd{\begin{description}} \def\ed{\end{description}}



\title{\bf Oversmoothing Tikhonov regularization   in Banach spaces }
\author{
De-Han Chen\footnote{School of Mathematics and Statistics $\&$ Hubei Key Laboratory of Mathematical
Sciences, Central  China Normal University, Wuhan, 430079, P.R.China. The
work of DC was financially supported by
National Natural Science Foundation of China
(Nos.  11701205 and 11871240). Email: dehan.chen@uni-due.de}\and
Bernd Hofmann\footnote{Faculty of Mathematics, Chemnitz University of Technology, 09107 Chemnitz, Germany.
The work of BH was financially supported by German Research Foundation (DFG grant HO 1454/12-1).  Email: hofmannb@mathematik.tu-chemnitz.de}
\and Irwin Yousept \footnote{Universit\"{a}t Duisburg-Essen, Fakult\"{a}t f\"{u}r Mathematik, Thea-Leymann-Str. 9, D-45127 Essen, Germany. The work of IY was financially supported by German Research Foundation (DFG grant YO159/2-2). Email:
 irwin.yousept@uni-due.de.
} \footnote{Author to whom any correspondence should be addressed.}
}
\begin{document}

\date{}
\maketitle

\begin{abstract}
This paper develops a   Tikhonov regularization theory for nonlinear ill-posed operator equations in Banach spaces. As the main challenge, we consider the so-called oversmoothing state  in the sense that the Tikhonov penalization is not able to capture the true solution regularity and leads to the infinite penalty value in the solution.  We establish a vast extension of the Hilbertian convergence theory      through the  use of   invertible sectorial operators from the   holomorphic functional calculus and   the prominent theory of interpolation scales in Banach spaces.  Applications of the proposed theory involving  $\ell^1$,   Bessel potential spaces, and Besov spaces are discussed.

\end{abstract}



\section{Introduction}

Tikhonov regularization is a    celebrated and powerful method for solving a wide class of nonlinear  ill-posed inverse problems of the type:  Given  $y \in Y$, find $x \in D(F)$ such that
\begin{equation} \label{eq:opeq}
F(x)=y,
\end{equation}
where  $F: D(F) \subseteq X \to Y$ is a nonlinear operator with domain $D(F)$ acting  between  two  Banach spaces $X$ and $Y$. Here, $D(F)$ is a closed and convex subset of $X$, and the forward operator $F$ is assumed to be weakly sequentially continuous.  Moreover,  for simplicity,  we suppose that the inverse problem \eqref{eq:opeq} admits a unique solution $\xdag \in D(F)$.   In the presence of   a  small perturbation  $y^\delta \in Y$   satisfying
\begin{equation} \label{ydelta}
\|y^\delta - y \|_Y \le \delta, \quad  \delta  \in (0,\delta_{\max}],
\end{equation}
for a  fixed constant $\delta_{\max}>0$,   the Tikhonov regularization method  proposes  a stable approximation  of the true solution $\xdag$ to   \eqref{eq:opeq}
by solving the following minimization problem:
\begin{equation}
\label{eq:Tik}
\left\{
\begin{aligned}
&\textrm{Minimize} \quad \Tda(x):=\|F(x)-\yd\|^{\nu}_Y+\kappa\|x\|_{V}^{m},\\
& \textrm{subject to} \quad    x\in \mathcal{D} := D(F) \cap V.
\end{aligned}
\right .
\end{equation}
In the setting of \eqref{eq:Tik},    $\kappa >0$ and $\nu,m\ge1$ are real numbers and  denote, respectively,  the Tikhonov regularization parameter and the exponents to the norms of the misfit functional and of the penalty functional. Furthermore, $V$ is a Banach space that is continuously embedded into $X$ with a strictly finer topology ($\exists z \in X: \, \|z\|_V=\infty$) such that the sub-level sets of the penalty functional $\|x\|_{V}^{m}$ are weakly sequentially pre-compact in $X$.  Thanks to this embedding $V\embed X$, the penalty $\|\cdot\|_V$ is stabilizing. Therefore,  by the presupposed conditions on the forward operator  $F:D(F) \subseteq X \to Y$, we obtain the existence and stability of  solutions  $x_\kappa^\delta$ to \eqref{eq:Tik} for all $\kappa>0$ (cf.~\cite[Section 4.1]{Schuster12} and \cite[Section~3.2]{Scherzer09}).

Throughout  this paper, the Tikhonov regularization parameter $\kappa$ is specified based on the following variant of the  discrepancy principle:
 For a  prescribed constant $C_{DP}>1$,   we choose  $\kappa=\adp>0$ in \eqref{eq:Tik} such that
\beq\label{dp}
\|F(\xda)-y^\delta\|_Y = C_{DP}\delta.
\eeq
In this paper,  let  $\delta_{\max}$ be sufficiently small, and we assume   that for all   $\delta \in (0,\delta_{\max}]$ and all  $y^\delta \in Y$ fulfilling  \eqref{ydelta},   there exist a   parameter   $\adp=\adp(\delta,y^\delta)>0$  and a solution $\xda$ to \eqref{eq:Tik} for   the regularization parameter $\kappa = \adp$ such that \eqref{dp} holds. If $F$ is linear, then  the condition
$$
\|y^\delta\|_Y>C_{DP}\delta
$$
is sufficient for the existence of  $\adp$ in the discrepancy principle \eqref{dp}.  However, due to possibly occurring duality gaps of the minimization problem (\ref{eq:Tik}), the solvability of   \eqref{dp} may fail to hold for     nonlinear operators $F$.   Sufficient conditions for the existence of $\adp$  can be found in \cite[Theorem~3.10]{AnzRam10}. The existence of  $\adp$ is assured in general whenever the minimizers  to \eqref{eq:Tik} are uniquely determined for all $\kappa>0$.

Unfortunately, our present techniques do not allow for a generalization of the discrepancy principle \eqref{dp} by replacing the  equality with an inequality. Our argumentations for the derivation of \eqref{VA} and \eqref{est5} are based on  the equality condition    \eqref{dp}.

The classical Tikhonov regularization theory relies on the fundamental assumption that the true solution $\xdag$ lies in the underlying penalization   space $V$. Under this requirement,  the  Tikhonov regularization method \eqref{eq:Tik} has been widely explored by many authors and  seems to have reached an advanced and satisfactory stage of mathematical development. In the real application, however, the   assumption $\xdag \in V$ often fails to hold since the (unknown)   solution regularity   generally cannot be predicted   a priori from the mathematical model. In other words, the so-called \emph{oversmoothing} state
\begin{equation} \label{eq:infty}
\Tda(\xdag)=\infty
\end{equation}
is highly possible to occur. For this reason, our present paper considers  the critical circumstance  \eqref{eq:infty}, which makes  the   analysis of   \eqref{eq:Tik} becomes  highly challenging and appealing at the same time.  In particular, the fundamental minimizing property $\Tda(x^\delta_\kappa) \le \Tda(\xdag)$ for any solutions to \eqref{eq:Tik}, used innumerably   in the classical regularization theory, becomes useless owing to (\ref{eq:infty}).

 Quite recently, motivated by the seminal paper \cite{Natterer84} for linear inverse problems, the second author and Math\'{e}  \cite{Hofmann18} studied the Tikhonov regularization method for nonlinear inverse problems with oversmoothing penalties in {\it Hilbert scales}. Their work  considers the quadratic case $\nu=m=2$ for \eqref{eq:Tik} and Hilbert spaces $X$, $Y$, and $V$.   Benefiting from the variety of link conditions in Hilbert scales,  they constructed auxiliary elements through  specific {\it proximal operators} associated with an auxiliary quadratic-type Tikhonov functional, which can be minimized in an explicit manner.  This idea leads to a convergence result for \eqref{eq:Tik} with a power-type rate. Unfortunately,  \cite{Hofmann18} and all its subsequent  extensions  \cite{GerthHof20,Hofmann20,HofHof20,HofPla20}  cannot be applied to the Banach setting or to the non-quadratic case  $m,\nu\neq 2$. First steps towards very special Banach space models for oversmoothing regularization have been taken recently in \cite{GerthHof19} and \cite[Section~5]{Miller20}.

Building on a  profound  application of sectorial operators  from the     holomorphic functional calculus (cf. \cite{Pazy}) and     the celebrated  theory of interpolation scales in Banach spaces (cf. \cite{Lunardi,Triebel}), our paper develops two novel convergence results  (Theorems \ref{main:sectorial} and \ref{main}) for the oversmoothing Tikhonov regularization problem \eqref{eq:opeq}-\eqref{eq:infty}. They substantially extend under comparable conditions the recent results for the Hilbert scale case to the  general Banach space setting. Furthermore, we are able to circumvent the technical assumption on  $\xdag$ being an interior point of $D(F)$ (see \cite{Hofmann18}) by  introducing an alternative invariance assumption, which serves as a remedy in the case of $\xdag \notin \textrm{int}(D(F))$.
We should underline that our theory is established  under a  two-sided nonlinear assumption \eqref{eq:F} on the forward operator. More precisely,   \eqref{eq:F}  specifies  that  for all $x \in D(F)$ the norm   $\|F(x)-F(\xdag)\|_Y$ is bounded from below and above by some  factors of $\| x-\xdag\|_{U}$ for a certain Banach space $U$,  whose topology is weaker than  $X$.  This condition       is    motivated by the Hilbertian case \cite[Assumption 2]{Hofmann18} and   seems to be reasonable  as it may characterize the degree of ill-posedness of the underlying inverse problem \eqref{eq:opeq}. On this basis,  Theorem \ref{main:sectorial} proves a convergence rate result for \eqref{eq:opeq}-\eqref{eq:infty} in the case where  $V$ is  governed by  an invertible  $\omega$-sectorial operator $A: D(A)\subset X \to X$ with a sufficiently small angle $\omega$ such that
\beq\label{eq:V}
 D(A)=V\quad \text{with norm equivalence}\quad
\|\cdot\|_V \sim \|A \cdot\|_X.
\eeq
The condition \eqref{eq:V} and the proposed invertible $\omega$-sectorial property allow us to apply the exponent laws and moment inequality (Lemma \ref{lemma:interp}), which are together with the  holomorphic functional calculus (Lemma \ref{lemma:func})   the central  ingredients for our proof.
It turns out   that our arguments   for our first result can be refined  by the theory of interpolation Banach  scales    with   appropriate decomposition operators  for the corresponding scales. This leads to our final result   (Theorem \ref{main}) which essentially generalizes    Theorem \ref{main:sectorial}.  In particular, Theorem \ref{main} applies to  the case where the underlying space $V$ cannot be described by an invertible sectorial operator satisfying \eqref{eq:V}. This occurs  (see Lemma \ref{lemma:example}), for instance, if the Banach space $X$ is reflexive (resp. separable),  but $V$ is non-reflexive (resp. non-separable).  Applications   and  examples with various Banach spaces, including $\ell^1$,   Bessel potential spaces, and Besov spaces,     are    presented and discussed in the final section.

This paper is organized as follows.  In the upcoming subsection,  we  recall some basic definitions and well-known facts regarding  interpolation couples and sectorial operators. The main results of this paper (Theorems \ref{main:sectorial} and \ref{main}) and all their mathematical requirements  are stated in Section \ref{sec:3}.  The proofs for these two results are presented, respectively, in  Sections \ref{sec:proof1} and \ref{sec:proof2}. The final section discusses  various applications of our theoretical findings, including those arising from inverse elliptic coefficient problems.

 Lastly, we mention that
 it would be desirable to replace the version \eqref{dp} of the discrepancy principle used throughout this paper by the {\it sequential discrepancy principle} (see, e.g.,~\cite[Algorithm~4.7]{HofPla20}), but to this more flexible principle there are not even convergence rates results for the oversmoothing case in the much simpler Hilbert scale setting (cf.~\cite{Hofmann18} where also only \eqref{dp} applies).
Therefore, such   extension is reserved for our future work.

\subsection{Preliminaries}\label{sec:pre}

	We begin by introducing   terminologies and notations used in this paper.    The space of all linear and bounded operators from $X$ to $Y$ is denoted by
	$$
	B(X,Y)= \{A: X \to Y \textrm{ is     linear and bounded}\},
	$$
	endowed with  the operator norm $\|A\|_{X\to Y}:= \sup_{\|x\|_X=1} \|Ax\|_{Y}$.
	If $X=Y$, then we simply write $B(X)$ for $B(X,X)$. The notation $X^*$ stands for the dual space of $X$.    The domain and range of a linear operator $A:D(A)\subset X\to Y$ is denote by $D(A)$ and $\textup{rg}(A)$, respectively. 
	A linear  operator $A:D(A) \subset X  \to X$
	is called closed if its graph $\{(x,Ax), ~ x \in D(A)\}$ is closed in  $X\times X$. If $A:D(A) \subset X  \to X$ is a linear and closed   operator,  then
	$$
	\rho(A) := \{\lambda \in \C   \mid    \lambda\textup{id} - A : D(A) \to X \textrm{ is  bijective and}\,\, ( \lambda\textup{id} - A )^{-1}\in B(X)\}
	$$
	and $$
	  \sigma(A) := \C \setminus \rho(A)
	$$
	denote, respectively,  the resolvent set and the spectrum of $A$. For every $\lambda \in \rho(A)$, the operator
	$
	R(\lambda,A):= (\lambda \textup{id} - A)^{-1} \in B(X)
	$
	is referred to as the resolvent operator of $A$. If $X$ is a Hilbert space, a densely defined operator $A:D(A)\subset X\to X$ is called self-adjoint if $A^* x=A x$ for all  $x\in D(A)$ and $D(A)=D(A^*)$. Moreover, if $X$ is a Hilbert space and
	$(Ax,x)_X> 0$ for all $x\in D(A) \setminus \{0\}$, then we say that $A:D(A)\subset X\to X$ is positive definite. If there exists a constant $c>0$, independent of $a$ and $b$, such that $c^{-1} a\leq b\leq c a $, we write $a\sim b$.  Throughout this paper,  we also make use of the set $\mathbb{N}_0:=\mathbb{N}\cup\{0\}$.

For two given Banach spaces $\mathcal X_1$ and $\mathcal X_2$, we call  $(\mathcal X_1,\mathcal X_2)$   an interpolation couple  if and only if there exists a locally convex topological space  $\mathcal U$ such that the embeddings
$$
 \mathcal X_1 \embed \mathcal U  \quad \text{and} \quad \mathcal X_2\embed \mathcal U
 $$
  are continuous. In this case, both $\mathcal X_1\cap \mathcal X_2$ and $\mathcal X_1+\mathcal X_2$ are well-defined Banach spaces. We say that  $\mathcal X_3$ is an intermediate space in $(\mathcal X_1,\mathcal X_2)$  if the embeddings
  $$
  \mathcal X_1\cap \mathcal X_2\embed \mathcal X_3\embed \mathcal X_1+\mathcal X_2
  $$
   are continuous.  Furthermore,   for   $s\in [0,1]$, we write  	
	$$
\mathcal X_3\in  J_s(\mathcal X_1,\mathcal X_2) \quad \iff \quad	\exists c\geq 0 \  \forall\, x\in \mathcal X_1\cap \mathcal X_2   : \  \|x\|_{\mathcal X_3}\leq c\|x\|_{\mathcal X_1}^{1-s}\|x\|_{\mathcal X_2}^s.
	$$
For a given interpolation couple $(\mathcal X_1,\mathcal X_2)$ and   $s\in[0,1]$,  $[\mathcal X_1,\mathcal X_2]_s$  denotes the complex  interpolation between $\mathcal X_1$ and $\mathcal X_2$.  For the convenience of the reader, we provide the definition of $[\mathcal X_1,\mathcal X_2]_s$ in the appendix. By a well-known result \cite[Proposition B.3.5]{Hasse},  we know that
	\beq\label{comp0}
	[\mathcal X_1,\mathcal X_2]_s\in J_s(\mathcal X_1,\mathcal X_2) \quad \forall s \in [0,1].
	\eeq
	On the other hand, for $s\in (0,1)$ and $q\in [1,\infty]$,  $(\mathcal X_1,\mathcal X_2)_{s,q}$  stands for the real  interpolation between $\mathcal X_1$ and $\mathcal X_2$ (see appendix for the precise definition). Similarly to \eqref{comp0}  {(see \cite[Corollary1.2.7]{Lunardi})}, it holds that
	\beq\label{compreal}
	(\mathcal X_1,\mathcal X_2)_{s,q}\in J_s(\mathcal X_1,\mathcal X_2) \quad \forall s \in (0,1) \quad \forall q\in [1,\infty].
	\eeq
\begin{Lemma}[see{\cite[Theorem B.2.3.]{Hasse}}]\label{lemma:interp-real}
Let $(\mathcal X_1, \mathcal X_2)$ and $(\mathcal Y_1,\mathcal Y_2)$ be interpolation couples. If
$T\in B(\mathcal X_1,\mathcal Y_1)\cap B(\mathcal X_2,\mathcal Y_2)$, then $T\in B((\mathcal X_1,\mathcal X_2)_{\tau,q},(\mathcal Y_1,\mathcal Y_2)_{\tau,q})$ for all $\tau \in (0,1)$ and $q\in [1,\infty]$. Moreover,
$$
\|T\|_{(\mathcal X_1,\mathcal X_2)_{\tau,q}\to (\mathcal Y_1,\mathcal Y_2)_{\tau,q}}\leq \|T\|^{1-\tau}_{\mathcal X_1\to \mathcal Y_1}
\|T\|^{\tau}_{\mathcal X_2\to \mathcal Y_2}.
$$
\end{Lemma}

 As described in the introduction, our theory is realized through the use of sectorial operators and  holomorphic functional calculus.  For the sake of completeness, let us recall some terminologies and well-known results regarding sectorial operators.
For $\omega\in (0,\pi)$, let
$$S_\omega:=\{z\in \mathbb{C}\backslash\{0\}\mid |\arg z|<\omega\}$$ denote the symmetric sector around the positive axis of aperture angle $2\omega$. If $\omega=0$, then we set $S_\omega=(0,+\infty)$.

\begin{Definition}[cf. {\cite{Hasse,kriegler2016paley}}] \label{def:sec}
	Let $\omega\in (0,\pi)$.  A  linear and closed operator $A:D(A)\subset X\to X$ is called $\omega$-sectorial  if the following conditions hold:
	\begin{enumerate}
		\item[(i)] the  spectrum $\sigma(A)$ is contained in $\overline{S_\omega}$.
		
		\item[(ii)] $\textup{rg}(A)$ is dense in $X$.
		
		\item[(iii)]
		$
		\forall\, \varphi\in (\omega,\pi) \,\, \exists \,\, C_\varphi>0 \,\, \forall \,\lambda\in \mathbb{C}\backslash \overline{S_\varphi}:
		\|\lambda R(\lambda,A)\|_{X\to X}\leq C_\varphi.
		$
	\end{enumerate}
	If $A:D(A)\subset X\to X$ is $\omega$-sectorial for all  $\omega \in (0,\pi)$, then it  is called $0$-sectorial.
\end{Definition}

If $X$ is a Hilbert space, and $A:D(A)\subset X\to X$ is positive definite and self-adjoint, then it is a $0$-sectorial operator  (see, e.g., \cite[Theorem 3.9.]{Pazy}). Moreover, a large class of second elliptic operators in general function spaces is also $\omega$-sectorial \cite[Section 8.4]{Pazy} for some $0\leq \omega<\frac{\pi}{2}$. Note that (ii) and (iii) imply that every $\omega$-sectorial operator is injective (cf. \cite{Hasse}). In the following, let $\omega\in (0,\pi)$ and $A:D(A)\subset X\to X$ be a  $\omega$-sectorial operator. For each  $\varphi\in (\omega,\pi)$,  we introduce the function space
$$
\mathcal H (S_\varphi):=\{f: S_\varphi \mapsto \mathbb{C}\, \textrm{is holomorphic} \mid \exists  C,\beta>0\, \forall z\in S_\varphi: |f(z)|\leq  C\min\{|z|^\beta, |z|^{-\beta}\}\}.
$$
We enlarge this algebra to
\beq\label{def:E}
\E(S_\varphi):=\mathcal H (S_\varphi) \oplus \textrm{Span}\{1\} \oplus \textrm{Span}\{\eta\}
\eeq
with $\eta(z):= (1+z)^{-1}$.
Given  a function $f\in \E(S_\varphi)$ with $f(z)=\psi(z)+c_1 + c_2\eta(z)$ for $c_1,c_2 \in \C$,   we   define
\beq\label{GAf}
G_A(f):=f(A):=\psi(A)+c_1\text{id}+ c_2(\textup{id}+A)^{-1}\in B(X)
\eeq
with $\psi(A)$  defined by the Cauchy-Dunford  integral
\beq\label{psi:integral}
\psi(A) :=\frac{1}{2\pi i}\int_{\Gamma_{\omega'}} \psi(z)R(z,A) dz,
\eeq
where $\Gamma_{\omega'}=\p S_{\omega'}$ denotes the boundary of the sector  $S_{\omega'}$ that is oriented counterclockwise and $\omega'\in (\omega,\varphi)$.   Note that the above integral is absolute convergent. Furthermore, by the Cauchy integral formula for vector-valued holomorphic functions, it admits the same value for all $\omega'\in (\omega,\varphi)$. Details of such construction can be  found in \cite{Hasse}.
\begin{Lemma}[see {\cite[Lemma 2.2.3]{Hasse}}]\label{Lemma:es}
Let $\varphi\in (0,\pi)$.
A holomorphic function $f: S_\varphi \to \mathbb{C}$ belongs to  $\E(S_\varphi)$  if and only if $f$ is bounded and has finite polynomial limits at $0$ and $\infty$, i.e.,  the limits $f_0:=\lim_{S_\varphi\ni z\to 0}f(z)$, $f_\infty:=\lim_{S_\varphi\ni z\to 0}f(z^{-1})$ exist in $\C$, and
$$
\lim_{S_\varphi\ni z\to 0}\f{|f(z)-f_0|}{|z|^\beta}=\lim_{S_\varphi\ni z\to 0 }\f{|f(z^{-1})-f_\infty|}{|z^{-1}|^\beta}=0
$$
for some $\beta>0$.
\end{Lemma}

There is a standard way to extend the functional calculus $G_A: \E(S_\varphi)\to B(X)$ to a larger algebra of functions on the sector $S_\varphi$  (see  \cite{Hasse}) with a larger range containing unbounded operators in $X$.   In particular, since $A$ is injective,  one can define {\it fractional power} $A^s$ for all $s\in \R$ by the extended functional calculus (see  \cite{Hasse}).  If $A:D(A)\subset X\to X$ is also invertible, then for any $s\geq 0$, the possibly unbounded operator $A^s:D(A^s)\subset X\to X$ is  an invertible operator with inverse $A^{-s}\in B(X)$ (see \cite[Proposition 3.2.3]{Hasse}). Therefore, $D(A^{-s})=X$, and
 {\it the fractional power domain space} $D(A^s)$ is a Banach space endowed with the norm
$\|\cdot\|_{D(A^s)}:=\|A^s\cdot\|_X$.  We collect the properties of the fractional power  operator $A^s$ that will be used below.

\begin{Lemma}[see {\cite[Propositions 3.2.1, 3.2.3 and 6.6.4.]{Hasse}}]\label{lemma:interp}
	If $A:D(A)\subset X\to X$ is an invertible $\omega$-sectorial operator for some $\pi>\omega\geq 0$, then the following assertions hold:
	
	\begin{enumerate}
		
		\item[\textup{ (i)}] For   $s_1,s_2\in \R$ with $s_1\geq s_2$, the embedding $D(A^{s_1})\embed D(A^{s_2})$ is continuous, and
		\beq\label{At}
		t^{s_1}A^{s_1}x=(tA)^{s_1} x\quad \forall  t>0 \, \text{and}\,\, x\in D(A^{s_1}).
		\eeq
		
		\item[\textup{ (ii)}] For   $s,t\in \R$,  it holds that $A^{t+s}x=A^{t}A^{s}x$ for all $x\in D(A^\tau)$ with $\tau=\max\{t,s,t+s\}$.
		
		\item[\textup{ (iii)}] If  $\beta\in (0,\pi/\omega)$, then $A^{\beta}$ is  invertible  $\omega\beta$-sectorial and for all $s>0$, we have $(A^\beta)^s x=A^{\beta s} x$ for all $x\in D(A^{\beta s})$. 		
		
		\item[\textup{ (iv)}] {\rm (Moment inequality)} For all
		$a\geq 0$, $s\ge 0$, and $-a\leq 0< r\leq s$,  there exists a constant $L>0$ such that
		\beq\label{Moment}
		\|A^r x\|_X\leq L\|A^s x\|_X^{\frac{r+a}{a+s}}\|A^{-a}x\|^{\f{s-r}{a+s}}_{X}\quad\forall\,x\in D(A^s).
		\eeq

	\end{enumerate}

\end{Lemma}
If $A:D(A)\subset X\to X$ is an invertible $\omega$-sectorial operator for some $\pi>\omega\geq 0$,  we     define the following Banach space:
\beq
X^s_A:=\begin{cases} (D(A^{s}),\|A^s\cdot\|_X)\quad &s\geq 0,\\
\text{Completion of}\, X \, \textup{under} \ \text{the norm}\ \|A^{s}\cdot\|_X\quad &s<0.
\end{cases}
\eeq
If { $X$ is a reflexive Banach space}, then  the adjoint operator $A^*:D(A^*)\subset X^*\to X^*$ is also an invertible $\omega$-sectorial operator \cite[Proposition 2.1.1 (d) and (j)]{Hasse}. Moreover,   for all $s\geq 0$ and $x\in X$, it holds that
\beq\label{frac:dual}
	\|A^{-s }x\|_{X}=\|x\|_{(X^{*,s}_{A})^*},
	\eeq
where $X^{*,s}_{A}$ denotes the fractional power domain $D((A^*)^s)$  (see \cite[Chapter V, Theorem 1.4.6]{Amann95} for more details).

\begin{Lemma}[see {\cite[Theorem 2.3.3]{Hasse}, \cite[Lemma 2.2 and 2.3]{Hasse05}}]\label{lemma:func}
Let $0\leq \omega<\pi$,  $A:D(A)\subset X\to X$ be an invertible $\omega$-sectorial operator, and $g\in \E(S_\varphi)$ for $\varphi\in (\omega,\pi)$.  Then, for any $s\in \R$,  it holds that
\beq\label{sect:commutative}
A^s g(tA) x= g(tA) A^s x \quad \forall t>0 \, \text{and}\,\, x\in D(A^s).
\eeq
Moreover,
\beq\label{sect:supCg}
C_g:=\sup_{t>0}\|g(tA)\|_{X\to X}<\infty
\eeq
and the mapping $t\mapsto f(tA)$ is continuous from $(0,\infty)$ to $ B(X)$.
 If, in addition,  the mapping $z\mapsto \psi(z):=z^s g(z) $  is of class $\E(S_\varphi)$ for some $s\in \R$,  then it holds that
 \beq\label{inclusion}
 \textup{rg}(g(tA))\subset D(A^s) \quad \forall t>0 \quad \textrm{if } s>0
 \eeq
 and
 \beq\label{eq:gA}
(tA)^s g(tA) x=\psi(tA)x \quad\textrm{for all } t>0 \textrm{ and all } x \in X.
\eeq
\end{Lemma}

\section{Main results}\label{sec:3}

We begin by formulating the  required two-sided nonlinear  mathematical  property for the forward operator $F:D(F) \subseteq X \to Y$:

\begin{Assumption}[Two-sided nonlinear structure] \label{ass1} There exist a Banach space $U \supsetneq X$ and
two numbers $0<c_U\leq C_U<\infty$ such that
\beq\label{eq:F}
c_U\|x-\xdag\|_{U}\leq \|F(x)-F(\xdag)\|_Y\leq C_U\|x-\xdag\|_{U}\quad \forall x\in D(F).
\eeq
{Moreover, there exists a neighborhood $B^\dag\subset X$ of $x^\dag$ such that
 the operator $F:  D(F) \subseteq X  \to Y$ is continuous in $B^\dag \cap D(F)$. }
\end{Assumption}
If the norm of the pre-image space is weakened to $\|\cdot\|_{U}$, i.e., ~if we
consider $U=X$, then the left-hand inequality of (\ref{eq:F})
implies that   (\ref{eq:opeq}) is {\it locally well-posed} at $\xdag$ (see \cite{HofPla18}). Of course, we do not consider the   case $U=X$  in (\ref{eq:F})   since the operator equation    \eqref{eq:opeq} is supposed  to be {\it locally ill-posed}.  Let us also note that the pre-image space    characterizes the  ill-posedness for the problem under the condition \eqref{eq:F} (see also \cite{Hofmann18} for further discussions).  In view of Assumption \ref{ass1},  \eqref{dp} yields the following result:

\begin{Lemma}\label{lemma:conv}
Let Assumption \ref{ass1} be satisfied and let the regularization parameter $\adp>0$ be chosen according to the discrepancy principle \eqref{dp}. Then,
\[
\|x^\delta_{\adp}-\xdag\|_{U}\leq \f{C_{DP} +1}{c_U}\delta
\]
holds for all $\delta \in (0,\delta_{\max}]$.
\end{Lemma}

\begin{proof}
In view of \eqref{ydelta}, \eqref{dp}, and  \eqref{eq:F},
$$
\|x^\delta_{\adp}-\xdag\|_{U}\leq \f{1}{c_U}\|F(x^\delta_{\adp})-F(\xdag)\|_Y\leq
\f{1}{c_U}(\|F(x^\delta_{\adp})-y^\delta\|_Y+\|y^\delta-y\|_Y)\leq \frac{(C_{DP}+1)\delta}{c_U}
$$
holds for all $\delta \in (0,\delta_{\max}]$  and all data $y^\delta$ obeying (\ref{ydelta}).
\end{proof}

Now we are ready to formulate the first main theorem. Its proof, the structure of which is analog to the proof of the theorem in \cite{Hofmann18}, will be given in Section~\ref{sec:proof1}.

\smallskip\smallskip

\begin{Theorem}\label{main:sectorial}
Suppose that   \eqref{eq:V}  and Assumption \ref{ass1} hold with an invertible $\omega$-sectorial operator  $A:D(A)\subset X\to X$  for  some angle $0\leq \omega <\pi$ and $U= X^{-a}_A$ for some  $a\geq 0$.  Furthermore, let the regularization parameter $\adp>0$ be chosen according to the discrepancy principle \eqref{dp}, and assume that there exists an
$f\in \E(S_\varphi)$ with  $\varphi \in (\omega,\pi)$  such that for every $s\in (0,1)$ the mappings
$z \mapsto z^{-(a+s)}(f(z)-1)$ and $z \mapsto z^s f(z)$ are of class $\E(S_\varphi)$. If the   solution $\xdag$ of  \eqref{eq:opeq} belongs to
$$
M_{\theta,E}:=\{x \in D(A^\theta)\mid \|A^\theta x\|_X\leq E\}
$$ for some $0<\theta<1$ and $E>0$ and satisfies either
\beq\label{sect:invariance}
\xdag \in  \text{int}(D(F))   \quad\textup{ or} \quad \exists \, t_0 \in (0,\infty) \  \forall\, 0<t\leq t_0 : \
f(tA) \xdag\in D(F),
\eeq
 then there exists a constant $c>0$, independent of   $\delta$, such that the error estimate
	\beq\label{conv:final}
	\|\xda-\xdag\|_{X} \le c\, \delta^{\f{\theta}{a+\theta}}
	\eeq
holds for all sufficiently small  $\delta>0$.
\end{Theorem}
\begin{Remark} The condition $\xdag \in \text{int}(D(F))$   means that   every $x \in X$ with a sufficiently small distance $\|x-\xdag\|_X$ belongs readily to $D(F)$.
If $\omega=0$, then the functions $f(z)=e^{-z^{a+1}}$ or $f(z)=(z^{a+1}+1)^{-1}$  satisfy  all the requirements
of Theorem \ref{main:sectorial} by Lemma \ref{Lemma:es}.  Moreover, if $\omega\in (0,\pi)$ and $a\geq 0$ are  sufficiently small, then  these two functions also satisfy   all the requirements  of Theorem \ref{main:sectorial}.

\end{Remark}

If both $X$ and $V$ are Hilbert spaces, then the condition   \eqref{eq:V}   with an invertible $0$-sectorial operator  $A:D(A)\subset X\to X$  is  valid if the embedding $V\embed X$ is, in addition, dense  (see Section \ref{subsect:hilbert} for more details).
We underline that     \eqref{eq:V}  is  the main restriction of   Theorem \ref{main:sectorial}  that could fail to hold in the practice, as the following lemma demonstrates:

\begin{Lemma}\label{lemma:example}
  Let $X$ be a reflexive (resp. separable) Banach space and $V$  non-reflexive (resp. non-separable). Then, there exists no  linear and closed operator $A: D(A)\subset X\to X$  satisfying  \eqref{eq:V}.
  \end{Lemma}
\begin{proof} Let us first consider the case where $X$ is  reflexive and $V$ is non-reflexive. We recall   the prominent Eberlein-$\breve{\textnormal{S}}$mulian theorem   that a Banach space is reflexive if and only if every   bounded sequence  contains a weakly converging subsequence.

Suppose   that there exists  a  linear and closed operator $A: D(A)\subset X\to X$ satisfying  \eqref{eq:V}. Let us consider the linear mapping
$$
P: D(A)  \to X\times X, \quad x \mapsto (x,Ax).
$$
By definition,   $P(D(A)) \subset X\times X$ is a closed subspace because $A: D(A)\subset X\to X$ is closed. Thus, since $X\times X$ is reflexive, $P(D(A))$ endowed with the norm $\|(x,Ax)\|_{P(D(A))}=\|x\|_X+ \|Ax\|_X$ is a reflexive Banach space (see  \cite[Theorem 1.22]{Adams}). On the other hand, thanks to  \eqref{eq:V} and $V \embed X$, both norms $\|\cdot\|_V$ and $\|\cdot\|_X+\|A \cdot\|_X$ are equivalent. Thus, the Eberlein-$\breve{\textnormal{S}}$mulian theorem  leads to a contradiction that  $\{V=D(A),\|\cdot\|_V\}$  is reflexive.

Let us next consider the case where $X$ is separable and $V$ is not separable. Suppose again that there exists  a  linear and closed operator $A: D(A)\subset X\to X$ satisfying  \eqref{eq:V}.  Then, as before,  since $P(D(A)) \subset X \times X$ is a closed subspace, and $X\times X$ is separable,  \cite[Theorem 1.22]{Adams}  implies that  $P(D(A))$ is  a separable Banach space. By the definition of separable spaces and since both norms $\|\cdot\|_V$ and $\|\cdot\|_X+\|A \cdot\|_X$ are equivalent, we obtain a contradiction that $\{V\!=\!D(A),\|\cdot\|_V\}$ is separable.
\end{proof}

Lemma \ref{lemma:example} motivates us to extend Theorem \ref{main:sectorial}  to the case where \eqref{eq:V} cannot be realized by  an invertible $\omega$-sectorial operator $A$. This   assumption is primarily required for the application of Lemma \ref{lemma:interp}. Therefore, it provides us with an illuminating hint of how to generalize the previous result by the theory of interpolation scales.

\begin{Assumption}[Interpolation scales] \label{scale} There exist a  Banach space $U$ satisfying Assumption \ref{ass1},  a family of  Banach spaces $\{X_s\}_{s\in [0,1]}$, and  a family of decomposition operators $\{P_t\}_{0<t\leq t_0} \subset B(U)$ with $t_0 \in (0,\infty)$ such that

\begin{enumerate}

\item[{\rm (i)}] The embeddings $V\embed X\embed U$ are continuous.

\item[{\rm (ii)}] $X_0=X$, $X_1=V$, and  the embedding $X_s\embed X_t$ is continuous for all $0\leq t\leq s\leq  1$.

\item[{\rm (iii)}] There exits a constant $a\geq 0$ such that for all $s\in (0,1]$ and   $r\in [0,s)$ it holds that
\beq\label{com-interp}
X_r\in J_{\f{s-r}{a+s}}(X_s,U) \quad \iff \quad \|x\|_{X_r}\leq L\| x\|_{X_s}^{\frac{r+a}{a+s}}\|x\|_{U}^{\frac{s-r}{a+s}}\quad \forall\,x\in X_s.
\eeq

\item[{\rm (iv)}] For any $s\in [0,1]$ and $t\in (0,t_0]$, it holds that  $P_t X_s\subset X_s$ with
	\beq\label{boud}
	C_P:=\sup_{0<t\leq t_0}\|P_t \|_{X_s\to X_s}<\infty,
	\eeq
        and  for any $x\in X$,  the mapping $t\to P_t x$ is continuous from $(0,t_0]$ into $ X$.
\item[{\rm (v)}]   For all  $0<s<1$, there exists a constant $C_{Proj}\geq C_P+1$   such that for all
	$0<t\leq t_0$
	\beq\label{proj}
	\|P_t -\textup{id} \|_{X_s\to U}\leq  C_{Proj} {t^{a+s}} \quad\text{and}\quad \|P_t \|_{X_s\to V}\leq  C_{Proj} t^{s-1}
	\eeq
	hold true with $a$ as in  \textup{(iii)}.
	\end{enumerate}
\end{Assumption}

The first three conditions    (i)-(iii) generalize the assumption of Theorem \ref{main:sectorial}  concerning the existence of an invertible $\omega$-sectorial operator  $A:D(A)\subset X\to X$. More precisely,  by Lemma \ref{lemma:interp} and  $X_s=X^{s}_A$, this assumption   implies  (i)-(iii), but not vice versa. On the other hand, (iv)-(v) weaken  the assumption of Theorem \ref{main:sectorial}  regarding the existence of the holomorphic function $f\in \E(S_\varphi)$.   Indeed,  as shown in   Section \ref{sec:proof1}, the linear operator  $P_t x:= f(t A)x$ satisfies the properties (iv)-(v). However, in general, the existence of  a family of  linear operators $\{P_t\}_{0<t\leq t_0} \subset B(U)$ fulfilling  (iv)-(v) does not   imply the existence of a holomorphic function $f\in \E(S_\varphi)$ satisfying the assumption of Theorem \ref{main:sectorial}.

Intuitively,  the decomposition property \eqref{proj} gives  a quantitative  characterization of the approximation  of elements in $V$ to an element $x\in X_s$.  Indeed,   \eqref{proj} shows that both  the ``distance'' between $x $ and $P_t x \in V$ in weaker norm and  the smoothness of $P_t x$ in $V$ can be controlled by  $\|x\|_{X^s_A}$. Similar properties have been utilized to verify variational source conditions for inverse PDEs problems in Hilbert spaces (see, e.g.,~\cite{chen2020,chenyousept2018}).

Let us finally state the regularity assumption for the true solution   $\xdag$ to  \eqref{eq:opeq}. Here,  in place of the exponent $A^\theta$ of the sectorial operator $A$,  we modify the smoothness condition of Theorem \ref{main:sectorial} by  using the scale of the Banach space $X$ and the corresponding  family   $\{P_t\}_{1<t\leq t_0}$ of decomposition operators from Assumption \ref{scale}.
\smallskip\smallskip

\begin{Assumption}[Solution smoothness]\label{ass:sol} There exist   $\theta\in (0,1)$ and $E>0$ such that
\beq\label{ME}
x^\dag\in M_{\theta,E}:=\{x\in X\mid \|x\|_{X_\theta}\leq E\},
\eeq
where $\{X_\theta\}_{\theta \in [0,1]}$ is  as in Assumption \ref{scale}.
Assume that one of the following conditions holds true:
\beq\label{interior}
\xdag \in \,\, \text{int}(D(F))  \quad\textup{ or} \quad
P_t \xdag\in D(F)\quad \forall \, 0<t\leq t_0,
\eeq
where $\{P_{t}\}_{0<t\leq t_0}$ is as in Assumption \ref{scale}.
 \end{Assumption}

 Now all assumptions are complete to formulate the second main theorem. Its proof will be given in Section~\ref{sec:proof2}

\begin{Theorem}\label{main}
Let Assumptions \ref{ass1}--\ref{ass:sol} be satisfied and and let the regularization parameter $\adp>0$ be chosen according to the discrepancy principle \eqref{dp}.  Then, there exists a    constant $c>0$, independent of  $\delta$, such that the error estimate
	\[
	\|\xda-\xdag\|_{X} \le c\, \delta^{\f{\theta}{a+\theta}}
	\]
	holds for all sufficiently small $\delta$.
\end{Theorem}

\section{Proof of Theorem \ref{main:sectorial}} \label{sec:proof1}
Let $ \delta \in (0,\delta_{\max} ]$ be arbitrarily fixed. In view of  the moment inequality  (Lemma \ref{lemma:interp}) for $s=\theta$ and $r=0$, it holds that
\[
\|\xda-\xdag\|_{X}\leq L \|A^{\theta}(\xda-\xdag)\|_{X}^{\f{a}{a+\theta}} \|A^{-a}(\xda-\xdag )\|_X^{\f{\theta}{a+\theta}}.
\]
Accordingly, we have  $U= D(A^{-a})$ and   $\|\cdot\|_{U} = \|A^{-a} \cdot\|_X$, which yields  that
\[
\|\xda-\xdag\|_{X}\leq L \|A^{\theta}(\xda-\xdag)\|_{X}^{\f{a}{a+\theta}} \| \xda-\xdag \|_U^{\f{\theta}{a+\theta}}.
\]
Therefore,   Lemma \ref{lemma:conv} implies that
\beq\label{xdacomp}
\|\xda-\xdag\|_{X}\leq L \|A^{\theta}(\xda-\xdag)\|_{X}^{\f{a}{a+\theta}} \left(\f{C_{DP}+1}{c_U}\delta\right)^{\f{\theta}{a+\theta}}.
\eeq
In conclusion, Theorem   \ref{main:sectorial} is valid, once we can
show the existence of a constant  $\hat{E}>0$, independent of   $\delta$,  such that
\beq\label{hatE}
\|A^{\theta}(\xda-\xdag)\|_{X}\leq \hat{E}.
\eeq

\noindent
{\bf Step 1}.  Let us define the approximate elements
\beq\label{xa}
{x}_t:=\begin{cases}f(t A)\xdag \quad &t>0, \\
x^\dag \quad &t=0.
\end{cases}
\eeq
In the following, we prove that there exists a constant $C_{ap}>0$,   depending only on $f$, such that \beq\label{conv:-a}
\|A^{-a}(\xa-x^\dag)\|_{X}\leq  C_{ap} E t^{{\theta+a}} \quad \forall t>0,
\eeq
\beq\label{conv:1}
\|A \xa \|_{X}\leq C_{ap} E t^{\theta-1}\quad \forall t>0,
\eeq
\beq\label{bound:P}
\|A^{\theta}(\xa-\xdag)\|_{X_\theta}\leq (1+C_{ap} )E\quad \forall t>0,
\eeq
\beq\label{conv:0}
\|\xa-\xdag\|_{X}\leq  C_{ap} E t^{\theta}\quad \forall t>0.
\eeq
In the following,  let $t>0$ be arbitrarily fixed.   Since $\xdag \in D(A^\theta)$ and $A^\theta \xdag \in X=D(A^{-\theta})$,
applying \eqref{sect:commutative}  and Lemma \ref{lemma:interp} (ii),
 we obtain that
\beq
A^{-a}(f(t A)\xdag-\xdag)=A^{-a}(f(t A)-\text{id})A^{-\theta}A^\theta \xdag=A^{-(a+\theta)}(f(t A)-\text{id}) A^\theta \xdag.
\eeq
As a consequence,
\begin{align}\label{sect:decomp1}
\|A^{-a}(\xa-\xdag)\|_X=
\|A^{-(a+\theta)}(f(t A)-\text{id}) A^\theta \xdag\|_X\notag\\
\underbrace{=}_{\eqref{At}} t^{a+\theta }\|(t A)^{-(a+\theta)}(f(t A)-\text{id}) A^\theta \xdag\|_X\notag
\underbrace{=}_{\eqref{eq:gA} \& \psi(z):=\f{(f(z)-1)}{z^{\theta+a}}} t^{\theta+a}\|\psi(t A)A^\theta \xdag\|_X\notag\\
\leq t^{\theta+a}\|\psi(t A)\|_{X\to X}\|A^\theta \xdag\|_X\leq t^{\theta+a} C_\psi \|A^\theta \xdag\|_X,
\end{align}
where we  used  \eqref{sect:supCg} with  $g=\psi$.  Notice that the function $\psi$ belongs to $\E(S_\varphi)$ due to our assumptions on $f$.
Similarly,  one has by { \eqref{sect:commutative},  and the Lemma \ref{lemma:interp} (ii) }that
 \begin{align*}
 \|A \xa \|_X=&\|A f(t A) A^{-\theta} A^\theta \xdag\|_X\underbrace{=}_{\eqref{inclusion}\textrm{ with } g=f}\|A^{1-\theta} f(t A)  A^\theta \xdag\|_X\\
 \underbrace{=}_{\eqref{At}} & t^{\theta-1}\|(t A)^{1-\theta} f(t A)  A^\theta \xdag\|_X\underbrace{=}_{\eqref{eq:gA} \textrm{ with } g=\widetilde \psi}t^{\theta-1} \|\widetilde{\psi}(t A) A^\theta \xdag\|_X,
\end{align*}
where $\widetilde{\psi}(z):=z^{1-\theta}f(z)$.   Due to our assumption on $f$, $\widetilde{\psi}$ belongs also to $\E(S_\varphi)$. Then, by setting $g=\widetilde{\psi}$ in  \eqref{sect:supCg},  we obtain from the above inequality that
\beq\label{sect:decomp2}
\|A \xa \|_X\leq t^{\theta-1}\|\widetilde{\psi}(t A)\|_{X\to X} \|A^\theta  \xdag\|_X\leq  t^{\theta-1} C_{\widetilde{\psi}}\|A^\theta \xdag\|_X.
\eeq
Next, a combination of  \eqref{sect:commutative} and  \eqref{sect:supCg} yields
\[
\|A^\theta x_t \|_X=\| f(t A)A^\theta \xdag\|_X\leq \| f(t A)\|_{X\to X}\|A^\theta \xdag\|_X\leq C_f \|A^\theta \xdag\|_X,
\]
which implies
\beq\label{sect:4}
\|A^\theta (\xa-\xdag) \|_X\leq (1+C_f)\|A^\theta \xdag\|_X.
\eeq
On the other hand,    the moment inequality  \eqref{Moment} with $s=\theta$ and $r=0$, we have
\begin{align}
\|x_t-\xdag\|_{X}\leq L \|A^{\theta}(x_t-\xdag)\|_{X}^{\f{a}{a+\theta}} \|A^{-a}(x_t-\xdag )\|^{\f{\theta}{a+\theta}} \label{sect:50} \\
\underbrace{\leq}_{\eqref{sect:decomp1} \& \eqref{sect:4}} L (1+C_f)^{\frac{a}{a+\theta}}C_\psi^\frac{\theta}{a+\theta} t^\theta  \|A^\theta \xdag\|_X. \label{sect:5}
\end{align}
In view of \eqref{sect:decomp1}-\eqref{sect:5},   the claims  \eqref{conv:-a}-\eqref{bound:P} follow due to  $\|A^\theta x\|_X\leq E$ and  by choosing $C_{ap}>0$ to be large enough.

\smallskip
\noindent
{\bf Step 2}. In this step,  we construct an important auxiliary element and study its basic properties.
Owing to  \eqref{sect:invariance}, \eqref{xa}, and \eqref{conv:0},   the approximate element $x_t$ belongs to $D(F)$ as long as $t$ is small enough.  {Thus,
thanks to  \eqref{conv:0} and { Lemma \ref{lemma:func}}, Assumption \ref{ass1} yields that  the mapping $t \mapsto \|F(x_t) - F(\xdag)\|_Y$ is continuous at all sufficiently small $t$ and converges to zero as $t \downarrow 0$.  For this reason, we  may reduce $\delta \in (0,\delta_{\max}]$ (if necessary) and find   a positive real number $t_{\textup{aux}}(\delta) >0$ satisfying}
\beq\label{def:aus}
x_{t_{aux}(\delta)} \in D(F) \quad \textrm{and} \quad (C_{DP}-1)\delta   = \|F(x_{t_{aux}(\delta)})-F(\xdag)\|_Y
\eeq
with    $C_{DP}>1$ as   in \eqref{dp}.  In all what follows,   we simply write  $x_{\text{aux}}(\delta):=x_{t_{\textup{aux}}(\delta)}$ for the auxiliary element. By \eqref{def:aus} and   \eqref{eq:F}, we obtain that
\beq\label{xaus:-a}
\|A^{-a}(\xaus(\delta)-\xdag)\|_{X} = \| \xaus(\delta)-\xdag\|_{U}  \leq \f{1}{c_U}\|F(\xaus(\delta))-F(\xdag)\|_Y= \f{C_{DP}-1}{c_U}\delta.
\eeq
Applying   \eqref{bound:P} and \eqref{xaus:-a} to  \eqref{sect:50} yields that
\beq\label{xaus:0}
\|\xaus(\delta)-\xdag\|_{X}\leq  L(1+C_{ap})^{a/(a+\theta)} E^{a/(a+\theta)}\left(\f{C_{DP}-1}{c_U}\delta\right)^{\f{\theta}{a+\theta}}.
\eeq
In addition, the auxiliary element also satisfies
 $$
 (C_{DP}-1)\delta\underbrace{=}_{\eqref{def:aus}}\|F(\xaus(\delta))-F(\xdag)\|_Y\underbrace{\leq}_{\eqref{eq:F}} C_U\|A^{-a}(\xaus(\delta)-\xdag)\|_{X}\underbrace{\leq}_{\eqref{conv:-a}}  C_{ap} E C_Ut_{\text{aux}}(\delta)^{{\theta+a}},
 $$
 which gives a low bound for $t_{\text{aux}}(\delta)$ as follows:
 $$
 t_{\text{aux}}(\delta)\geq \left(\f{C_{DP}-1}{C_{ap} E C_U}\delta\right)^{\f{1}{a+\theta}}.
 $$
Making use of this lower bound and due to $0<\theta <1$, we eventually obtain
 \beq\label{xaus:1}
\|A \xaus(\delta)\|_{X}\underbrace{\leq}_{\eqref{conv:1} }  C_{ap} E t_{\textup{aux}}(\delta)^{{\theta-1}} \le   (C_{ap} E)^{\f{a+1}{a+\theta}} \left(\f{C_{DP}-1}{C_U}\delta\right)^{\f{\theta-1}{a+\theta}}.
\eeq

\smallskip
\noindent
{\bf Step 3}. In this step, we prove  \eqref{hatE}. According to \eqref{bound:P},
the auxiliary element $\xaus(\delta)$ satisfies $\|A^{\theta}(\xaus(\delta)-\xdag)\|_{X}\leq (1+C_{ap})E$ for all $\delta\in (0, \delta_{\max}]$. Therefore, we see that  \eqref{hatE} is valid  if we are able to prove that the existence of constant $E'>0$,  independent of $\xdag,\xda, \delta$, and $\kappa$, such that
$$
\|A^\theta(\xda-\xaus(\delta)) \|_{X}\leq E'.
$$
Since  $\xda$ is a  minimizer for the Tikhonov regularization problem \eqref{eq:Tik} and $ \xaus(\delta) \in \mathcal{D} = D(F) \cap V$ (due to \eqref{def:aus}, \eqref{eq:V}, and \eqref{xaus:1}),we have
\begin{align*}
(C_{DP}\delta)^{\nu}+\kappa \| \xda \|^{m}_{V}&\underbrace{=}_{\eqref{dp}}\|F(\xda)-y^\delta\|_Y^{\nu}+\kappa  \|\xda\|^{m}_{V}\\
&\, \, \leq  \|F(\xaus(\delta))-y^\delta\|_Y^{\nu}+\kappa  \| \xaus(\delta)\|^{m}_{V}\\
&\underbrace{\leq}_{\eqref{def:aus}}((C_{DP}-1)\delta+\delta)^{\nu}+\kappa  \|\xaus(\delta)\|^{m}_{V},
\end{align*}
which affirms that
\beq\label{VA}
\| \xda\|_{V}
\le \| \xaus(\delta) \|_{V} \quad \underbrace{\implies}_{\eqref{eq:V}}\quad  \|A \xda\|_X\leq C_A \|A \xaus(\delta) \|_{X}
\eeq
with $C_A:=\max\{\|\textup{id}\|_{D(A)\to V}, \|\textup{id}\|_{V\to D(A)}\}^2$. The above inequality, along with the triangle inequality, implies
\beq\label{VA2}
\|A(\xda-\xaus(\delta))\|_{X}\le \|A\xda \|_{X}+\| A\xaus(\delta)\|_{X}
\leq (1+C_A) \| A\xaus(\delta) \|_{X}.
\eeq
Now, applying \eqref{Moment} with $s=1$ and $r=\theta$, we obtain
\begin{align}\label{est3}
\|A^\theta(\xda-\xaus(\delta))\|_{X}\leq  L \|A(\xda-\xausd)\|^{\f{\theta+a}{a+1}}_{X} \|A^{-a}(\xda-\xausd)\|^{\f{1-\theta}{a+1}}_{X}\notag\\
\underbrace{\leq}_{\eqref{VA2}} L (C_A+1)^{\f{\theta+a}{a+1} } \|A \xaus(\delta) \|_{X}^{\f{\theta+a}{a+1}} \|A^{-a}(\xda-\xausd)\|_{X}^{\f{1-\theta}{a+1}}.
\end{align}
The first factor in the right-hand side of \eqref{est3} can be estimated by \eqref{xaus:1} as follows:
\beq\label{est4}
 \|A \xaus(\delta) \|_{X}^{\f{\theta+a}{a+1}}\leq
  C_{ap} E \left(\f{C_{DP}-1}{C_U}\delta\right)^{\f{\theta-1}{a+1}},
\eeq
whereas the second factor can be estimated as follows
\begin{align}\label{est5}
 \|A^{-a}(\xda-\xausd)\|_{X}\leq&  \|A^{-a}(\xda-\xdag)\|_{X}
 + \|A^{-a}(\xausd-\xdag)\|_{X} \notag \\
 = & \|\xda-\xdag\|_{U} + \|A^{-a}(\xausd-\xdag)\|_{X} \notag\\
&\hspace{-1cm} \underbrace{\leq}_{\eqref{eq:F} \& \eqref{xaus:-a}}\hspace{-0.3cm} \f{1}{c_U}\|F(\xda)-F(\xdag)\|_{Y}+ \f{C_{DP}-1}{c_U}\delta
  \underbrace{\leq}_{\eqref{ydelta} \&\eqref{dp}} \f{2C_{DP} }{c_U}\delta.
\end{align}
In conclusion,  \eqref{est3}-\eqref{est5} yield
\begin{align*}
\|A^{\theta}(\xda-\xaus(\delta))\|_{X}\leq L (C_A+1)^{\f{\theta+a}{a+1} } C_{ap} E  \left(\f{C_{DP}-1}{C_U}\right)^{\f{\theta-1}{a+1}}  \left(\f{2C_{DP} }{c_U}\right)^{\f{1-\theta}{a+1}} =: E'.
\end{align*}
This completes the proof.  \qed

\section{Proof of Theorem \ref{main}} \label{sec:proof2}

We now generalize the arguments used in   the previous section for the proof of Theorem \ref{main}.   Our goal is to prove the existence of a constant $ \hat{E}_*>0$, independent of $\delta$,  such that
\beq\label{hatE2}
\|\xda-\xdag\|_V\leq \hat{E}_*
\eeq
holds true for all sufficiently small $ \delta$. This estimate implies the claim of Theorem \ref{main}, since \eqref{hatE2} together with  Lemma \ref{lemma:conv} and \eqref{com-interp}  for $r=0$ and $s=\theta$  implies
\beq
\|\xda-\xdag\|_{X}\leq L\hat{E}_*^{\f{a}{a+\theta}} \left(\f{C_{DP}+1}{c_U}\delta\right)^{\f{\theta}{a+\theta}}.
\eeq
Let $ \delta \in (0,\delta_{\max} ]$ be arbitrarily fixed.
We define
\beq\label{xa2}
\xah:=\begin{cases} P_t \xdag \quad   &t \in (0,t_0],\\
x^\dag \quad &t=0.
\end{cases}
\eeq
According to  \eqref{proj} with $s=\theta$ and \eqref{ME},  it holds that
\beq\label{final1}
\|\xah-\xdag\|_U\leq C_{Proj} t^{a+\theta}E\quad \text{and}\quad \|\xah\|_V\leq C_{Proj}t^{1-\theta}E.
\eeq
On other hand,  Assumption \ref{scale}  (iv) ensures that
\beq\label{final2}
\|\xah\|_{X_\theta}\leq C_P\|\xdag\|_{X_\theta}\underbrace{\leq}_{\eqref{ME}}
C_PE,
\eeq
which implies
\beq\label{final3}
\|\xah-\xdag\|_{X_\theta}\underbrace{\leq }_{\eqref{ME}}
(C_P+1)E\leq C_{Proj} E.
\eeq
A combination of \eqref{com-interp}, with $r=0$ and $s=\theta$, \eqref{final1}, and \eqref{final3} yields
\beq\label{final4}
\|\xah-\xdag\|_{X}\leq L\|\xah-\xdag\|_{X_\theta}^{\f{a}{a+\theta}}\|\xah-\xdag\|_{U}^{\f{\theta}{a+\theta}}\leq L C_{Proj}t^\theta E.
\eeq
In view of \eqref{interior}, \eqref{xa2}, \eqref{final1}, and \eqref{final4},  $\hat{x}_t\in \mathcal D=D(F)\cap V$ holds true  as long as $t$ is small enough.  Now, if necessary, we   reduce $\delta$ to obtain an auxiliary element  $\xaush(\delta)\in \mathcal D$  satisfying \eqref{def:aus}
with $\xa$ replaced by $\xah$. Proceeding as in  {\bf Step 2} of the proof of Theorem \ref{main:sectorial} (see \eqref{xaus:-a} and \eqref{xaus:1}),  by \eqref{final1},
it follows that
\begin{align}
\|\xaush(\delta)-\xdag\|_{U}&\leq  \f{C_{DP}-1}{c_U}\delta,\label{ff1}\\
\| \xaush(\delta)\|_{V}&\leq  {(C_{Proj} E)^{^{\f{a+1}{a+\theta}} } }\left(\f{C_{DP}-1}{C_U}\delta\right)^{\f{\theta-1}{\theta+a}}.\label{ff2}
\end{align}
By similar arguments for \eqref{VA}, we also obtain
 \beq
\|\xda\|_V\leq \|\xaush(\delta)\|_V,
 \eeq
and consequently  \eqref{com-interp} with $r=\theta$ and $s=1$ implies that
\begin{align}\label{final6}
\|\xda-\xaush(\delta)\|_{X_\theta}\leq& L 2^{\f{\theta+a}{a+1}}\|\xaush(\delta)\|_V^{\f{\theta+a}{a+1}}
\|\xda-\xaush(\delta)\|_{U}^{\frac{1-\theta}{a+1}}\notag\\
\underbrace{\leq}_{ \eqref{ff2}}& L 2^{\f{\theta+a}{a+1}}  { C_{Proj} E}
 \left(\f{C_{DP}-1}{C_U}\delta\right)^{\f{\theta-1}{a+1}}\|\xda-\xaush(\delta)\|_{U}^{\frac{1-\theta}{a+1}}.
\end{align}
Similar as in \eqref{est5},  applying \eqref{eq:F} and \eqref{ff1} results in
$$
\|\xda-\xaush(\delta)\|_{U}\leq \frac{2C_{DP}}{c_U}\delta.
$$
Thus, inserting the above inequality  into \eqref{final6}, we conclude that
 the desired estimate \eqref{hatE2} holds with
$$
\hat{E}_*=
L 2^{\f{\theta+a}{a+1} }   C_{Proj} E \left(\f{C_{DP}-1}{C_U}\right)^{\f{\theta-1}{a+1}} \left(\f{2C_{DP}}{c_U}\right)^{\f{1-\theta}{a+1}}.
$$
This completes  the proof. \qed

\section{Applications}\label{sect:decomp}

 In this section, we present some applications of the abstract theoretical results from Theorems \ref{main:sectorial} and \ref{main}. The first three applications present different possible choices and settings for the governing Tikhonov-penalties: Hilbertian case, $\ell^1$-penalties,  and Besov-penalties. While these three examples are somewhat still too abstract, we present a more practical application of our theory in Section \ref{sec:app} regarding radiative problems \cite{engl89,engl2000,bangti11}. We also believe that our theory is applicable to the Tikhonov regularization method for nonlinear electromagnetic inverse or design problems arising for instance  in the context of ferromagnetics \cite{Yousept2013,LamYousept2020}, superconductivity \cite{Yousept2017,Yousept2020a}, electromagnetic shielding \cite{Yousept2020b}, and many others. Such problems   suffer from oversmoothing phenomena, mainly due to the low regularity and the lack of compactness properties in the associated function space for the electromagnetic fields. Another related application with oversmoothing character   can be found in elastic   full waveform inversion.  The application of our theory to all these problems  require however   investigations with different techniques  that would go beyond the scope of our present paper and will be considered in our upcoming research.

\subsection{Hilbertian penalties}\label{subsect:hilbert}

Let us first consider the case where both $X$ and $V$ are Hilbert spaces, and the embedding
	$V\embed X$ is dense and continuous.  By the Riesz representation theorem, there exists an isometric isomorphism
\beq
B_V:V\to V^*, \quad \la B_V u, v\ra_{V^*\times V}=(u,v)_{V} \quad \forall v\in V,
\eeq
which induces an unbounded operator $B_X:D(B_X)\subset X\to X$ by
\beq\label{A:Hilbert}
B_X u= B_Vu \quad \text{with} \,\,D(B_X)=\{u\in V\mid B_Vu\in X\}.
\eeq
\begin{Lemma}[{see \cite[{Theorems 2.1 and 2.34}]{Yagi}}]\label{pro:A}
	If  both $X$ and $V$ are Hilbert spaces such that the embedding
	$V\embed X$ is dense and continuous, then the operator $B_X:D(B_X)\subset X\to X$ defined by \eqref{A:Hilbert} is an invertible positive definite and self-adjoint operator. Moreover, its square root $A=B_X^{1/2}: D(A)\subset X\to X$ of $B_X$  fulfills
	$$
	D(A)=V \quad \text{with norm equivalence}\quad \|\cdot\|_V \sim \|A \cdot\|_X.
	$$
	\end{Lemma}
The above-defined operator $A: D(A)\subset X\to X$ is also  an invertible, positive definite, and self-adjoint operator. Therefore, \eqref{eq:V} is true for this case. Moreover,  there exits $\lambda_0>0$ such that  $\sigma(A)\subset [\lambda_0,\infty)$, and for any $\omega\in (0,\pi)$, the following estimate holds
\beq\label{resolvent}
\|R(z,A)\|_{X\to X}\leq \frac{1}{\text{dist}(z,\sigma(A))}\leq \frac{ 1}{ |z|\sin \omega} \,\quad \forall z\in \C\backslash \overline{S_{\omega}}\quad \forall\,\omega\in (0,\pi).
\eeq
On the other hand, it is well-known that there exists an spectral resolvent $\{E_\lambda\}_{\lambda\geq \lambda_0}$ for $(A,D(A))$. Let $\mathcal{B}$ be the algebra of all Borel measurable function over  $[\lambda_0,\infty)$, and $\mathcal{B}_\infty$ be the sub-algebra of $\mathcal{B}$ consists of all essentially bounded functions. Then, for every $f \in \mathcal{B}_\infty$, we can define an algebra homomorphism from
$\mathcal{B}$ into closed operators on $X$ by
$$
f(A):=\int_{\lambda_0}^{\infty} f(\lambda ) d E_\lambda x \quad \forall\, x\in D(f(A))
=\{x\in X\mid \int_{\lambda_0}^{\infty} |f(\lambda )|^2 d \| E_\lambda x \|^2<\infty \},
$$
satisfying  $\|f(A)\|_{X\to X}\leq \|f\|_\infty$. Furthermore,   a simplified  analogue of Lemma \ref{lemma:func} is obtained as follows:

\begin{enumerate}

\item[\textrm(a)] If $f\in \mathcal B_\infty$ and $s\geq 0$, then $A^s f(A)x=f(A) A^s x$ for all $x\in D(A^s)$.

\item[\textrm(b)]  If $f\in \mathcal B_\infty$, then $\sup_{t>0}\|f(tA)\|_{X\to X}\leq \|f\|_\infty$.

\item[\textrm(c)]  If $f \in \mathcal B_\infty$ and $\lambda\mapsto \lambda^s f(\lambda)$ belongs to $ \mathcal B_\infty$  for some $s>0$, then
$\|A^s f(A)\|_{X\to X}<\infty$.

\end{enumerate}
Therefore, in the Hilbertian  setting,  Theorem \ref{main:sectorial}   remains true
if we replace the condition $ f \in \E(S_\varphi)$  by  $ f \in \mathcal{B}_\infty$ being continuous.  This can be seen as a generalization  of Theorem \ref{main:sectorial} because $\mathcal{B}_\infty$ is larger than  $ \E(S_\varphi)$.
 As an instance,  we can choose the following non-holomorphic function
$$
f(\lambda)=\begin{cases}
1 \quad &\text{if } \lambda \in [0,1],\\
2 - \lambda  \quad &\text{if } \lambda \in [1,2],\\
0,\quad &\text{if } \lambda\geq 2.
\end{cases}
$$
We underline that the generalization of Theorem \ref{main:sectorial} in the Hilbert case using a continuous but not necessarily holomorphic function $f \in \mathcal{B}_\infty$ is important since
the  choice of $f$ influences the invariance requirement in \eqref{sect:invariance}.

\subsection{$\ell^1$-penalties}

In this section, we consider the case when $V=\ell^1$ but the solution $\xdag$ to \eqref{eq:opeq} lies in  $X=\ell^p(w)$ (see  Definition \ref{def:wl}) with $1<p<\infty$. Thus, in view of Lemma \ref{lemma:example},   \eqref{eq:V}  fails to hold such that Theorem \ref{main:sectorial} does not apply to this case. {  In the case of $V=\ell^1$,   one typically sets $m=1$ for the exponent of the penalty functional (cf. \cite{GerthHof19}), whereas the exponent $\nu \ge 1$ for the misfit functional may vary  depending on the  mathematical model.}

\begin{Definition}\label{def:wl}
Let $w: \mathbb{N}_0\to (0,\infty)$ be a  function.  For any $1\leq p<\infty$, we define
\beq
\ell^p(w):=\left\{\{x(n)\}_{n=0}^{\infty}\mid
\|x\|_{\ell_w^p}:=(\sum_{n\in \mathbb{N}_0} w(n)^{1-p}|x(n)|^p)^{1/p}<\infty \right\}
\eeq
and  for $p=\infty$
$$
\ell^\infty (w):=\left\{\{x(n)\}_{n=0}^{\infty}\mid
\|x\|_{\ell^\infty(w)}:=\sup_{n\in \mathbb{N}_0}|w(n)^{-1}x(n)|<\infty \right \}.
$$

\end{Definition}
If $p=1$, then $\ell^p(w)$ identical to $\ell^1$.  Also, we would like to mention that
$\ell^p(w)$ is reflexive if $1<p<\infty$.  In the following lemma, we verify Assumption \ref{scale}   for  Theorem \ref{main}.

\begin{Proposition}

Let $V=\ell^1$ and $w: \mathbb{N}_0\to (0,\infty)$   such that  $w(n_1) \le w(n_2)$ for all $n_1 < n_2$ with $\lim_{n\to \infty}w(n)=\infty$.   Furthermore, suppose   that there exists a non-negative, decreasing  function $f\in C^1[0,\infty)$ such that  $f(0)=1$, and
	\beq\label{growth}
 \sum_{n\in \mathbb{N}_0} w(n) f(\tau w(n)) \leq C_w \tau^{-1}  \quad \forall \tau \in (0,\tau_0]
\eeq
holds with some real numbers $\tau_0>0$ and $C_w>0$.  If $1< p<\infty$ and  $X=\ell^p(w)$, then Assumption \ref{scale}  holds with
$U=\ell^\infty(w)$, $a=\frac{1}{p-1}$,
$$
X_s:=\ell^{p_s}(w)\quad \text{with}\quad p_s:=\f{p}{1+s(p-1)}\quad s\in[0,1],
$$
and
\beq\label{ellproj}
(P_t x)(n):= f(t^{a+1} w(n)) x(n)\quad \forall  x\in \ell^{p}(w)
\eeq
for all $t\in (0,t_0]$ with $t_0:=\tau_0^{\f{1}{1+a}}$.

\end{Proposition}

\begin{proof}

For any $q\geq 1$, it holds that
\beq\label{int1}
\|x\|_{\ell^\infty(w)}\leq \frac{1}{w(0)} \|x\|_{\ell^{q}(w)}\quad \forall x\in \ell^{q}(w),
\eeq
since $w(n)\geq w(0)$ for all $n\geq 0$.
Furthermore, if $1\leq q_1\leq q_2$,  then the embedding $\ell^{q_1}(w)\embed \ell^{q_2}(w)$ is continuous since
\begin{align}\label{int2}
\|x\|_{\ell^{q_2}(w)}&=(\sum_{n=0}^{\infty} w|w^{-1} x|^{q_2})^{1/q_2}
= (\sum_{n=0}^{\infty} w|w^{-1} x|^{q_1} |w^{-1} x|^{q_2-q_1})^{1/q_2}
\leq \|x\|_{\ell^{q_1}(w)}^{\f{q_1}{q_2}}\|x\|_{\ell^\infty(w)}^{1-\frac{q_1}{q_2}}\notag\\
&\leq w(0)^{q_1/q_2-1} \|x\|_{\ell^{q_1}(w)} \quad  \forall x\in \ell^{q_1}(w),
\end{align}
where we have used \eqref{int1} with $q=q_1$. Both \eqref{int1} and \eqref{int2} verify the conditions (i)-(ii) of Assumption \ref{scale}.
For any $s\geq 0$ and $r\in [0,s]$, let us choose $q_1=p_s$ and $q_2=p_r$ in \eqref{int2}, which yields due to $p=1+\frac{1}{a}$ that
\begin{align}
\|x\|_{\ell^{p_r}(w)}\leq \|x\|_{\ell^{p_s}(w)}^{\frac{1+r(p-1)}{1+s(p-1)}}
\|x\|_{\ell^\infty(w)}^{\frac{(p-1)(s-r)}{1+s(p-1)}}
= \|x\|_{\ell^{p_s}(w)}^{\frac{a+r}{a+s}}
\|x\|_{\ell^\infty(w)}^{\frac{s-r}{a+s}}\quad  \forall x\in \ell^{p_s}(w)=X_s.
\end{align}  This verifies the condition (iii) of Assumption \ref{scale}.

Obviously,  from \eqref{ellproj}  and the property that $\displaystyle \max_{\tau \in [0,\infty)}|f(\tau )|= 1$,
it follows   that
$$
|(P_t x) (n)|\leq |x(n)|\quad \forall\, t\in (0,t_0],  \,\, n\in \mathbb N_0, \,\text{and}\, x\in \ell^{p_s}(w)=X_s,
$$  which implies that  $\|P_t x\|_{X_s\to X_s}\leq 1=: C_p$ for all $t\in (0,t_0]$ and all $s\in [0,1]$.  From the continuity of
$f$ and $\displaystyle \max_{\tau \in [0,\infty)}|f(\tau )|=1$,   it follows by  \eqref{ellproj}
 that for every $x\in X=\ell^p(w)$, the mapping $t\mapsto P_t x$ is continuous from $(0,t_0]$ into $X$. Therefore, the requirement \eqref{boud}
is satisfied.

On the other hand, for every $s\in (0,1)$,  it follows from the (right) differentiability of $f$ at $0$  that  there exists a constant $C^*(s)$, only depending on $s$, such that
$$
\f{1-f(\tau^{p_s})}{\tau}=\f{f(0)-f(\tau^{p_s})}{\tau} \leq C^*(s)\quad \forall \tau>0,
$$
since $p_s>1$. By plugging $\tau= t^{a+s}w(n)^{\f{1}{p_s}}$ in the above inequality and using $p_s=\f{a+1}{a+s}$, we obtain
$$
1-f(t^{a+1}w(n))\leq C^*(s)  t^{a+s} w(n)^{\f{1}{p_s}}\quad \forall\, n\in \mathbb{N}_0.
$$
Hence, for all  $x\in \ell^{p_s}(w)=X_s$ and $n\in \mathbb{N}_0 $, it holds  that
\begin{align*}
w(n)^{-1}|(1-f(t^{a+1} w(n)))x(n)| \leq&    C^*(s)  t^{a+s} w(n)^{\f{1}{p_s}-1}|x(n)|\\
\leq&  C^*(s)  t^{a+s} (\sum_{n\in \mathbb{N}}
w(n)^{1-p_s}|x(n)|^{p_s})^{1/p_s},
\end{align*}
which ensures that
\begin{align*}
\|(\text{id}-P_t) x\|_{\ell^\infty(w)} \leq  C^*(s)   t^{a+s}\|x\|_{\ell^{p_s}(w)}.
\end{align*}
 On the other hand, H\"{o}lder's inequality implies
   that
 \begin{align*}
 \|P_t &x\|_{\ell^1}=\sum_{n\in\mathbb N_0}f(t^{a+1} w(n)) |x(n)|
 =\sum_{n\in\mathbb N_0} f(t^{a+1} w(n))  w(n)^{1-\frac{1}{p_s}}  w(n)^{\f{1}{p_s}-1}|x(n)| \\
 &\underbrace{\leq}_{1-\frac{1}{p_s}=\frac{1-s}{a+1}}
\left(\sum_{n\in\mathbb N_0}w(n)f(t^{a+1} w(n))^{\f{a+1}{1-s}}  \right)^{\frac{1-s}{a+1}}\left(\sum_{n\in\mathbb N_0}w(n)^{1-p_s}| x_n|^{p_s}\right)^{1/p_s}\\
&\,\,\,\,\underbrace{\leq}_{f \leq 1}
\left(\sum_{n\in\mathbb N_0}w(n)f(t^{a+1} w(n))  \right)^{\frac{1-s}{a+1}}\left(\sum_{n\in\mathbb N_0}w(n)^{1-p_s}| x_n|^{p_s}\right)^{1/p_s}\\
&\quad\leq C^{\f{1-s}{a+1}}_w t^{s-1}\|x\|_{\ell^{p_s}(w)},
 \end{align*}
 where we have used the growth rate  \eqref{growth} with $\tau=t^{a+1}$. In conclusion, the last condition \eqref{proj}  holds true.
\end{proof}

{
\subsection{Besov-penalties}

For any $s\in \R$ and $1<p<\infty$, we define the Bessel potential space
$$
H^s_p(\R^d):=\{u\in \mathcal{S}(\R^d)'\mid \|u\|^p_{H^s_p(\R^d)}:=\int_{\R^d}| \int_{\R^d} e^{i\xi x}\la \xi\ra^s \hat{u}(\xi) d\xi |^p dx<\infty  \},
$$
where $\la \xi\ra:= (1+|\xi|^2)^{\frac{1}{2}}$, $\hat{u}:=\F(u)$,  $\mathcal{F}:\mathcal{S}(\R^d)'\to \mathcal{S}(\R^d)'$ is the Fourier transform, and $\mathcal{S}(\R^d)'$ denotes
the tempered distribution space  (see, e.g.,  \cite{Yagi}). If $s$ is a non-negative integer,
$H^s_p(\R^d)$ is identical to the classical Sobolev space $W^{s,p}(\R^d;\mathbb{C})$. In particular,
$L^p(\R^d)=H^0_p(\R^d)$ is the space of  complex-valued $p$-integrable functions.
Throughout this subsection, let us define the operator $A_p:H^{1}_p(\R^d)\subset L^p(\R^d)\to L^p(\R^d)$ given by
$$
A_p u:=\mathcal{F}^{-1}(\la \xi\ra \hat{u}(\xi)):=\sqrt{(I-\Delta_p)}\, u,
$$
where $-\Delta_p: H^2_p(\R^d)\subset L^p(\R^d)\to L^p(\R^d)$ denotes the Laplace operator on $L^p(\R^d)$.
By a well-known result \cite[Theorem 8.2.1]{Hasse05}, it is an
invertible 0-sectorial operator. Moreover, for any holomorphic function
$f\in \E(S_\varphi)$ with $0<\varphi<\pi$,  the operator $f(A_p)\in B(L^p(\R^d))$ admits the characterization:
\beq\label{Apf}
f(A_p)u=\F^{-1}(f(\la \xi\ra ) \hat{u}(\xi))\quad \forall\, u\in L^p(\R^d),
\eeq
i.e., $f(A_p)$ can be characterized by the {\it Fourier multiplier} associated with the function $\R^d\backslash\{0\}\ni \xi\mapsto f(\la \xi \ra)$ (see e.g. \cite[Proposition 8.2.3]{Hasse}).
 In all what follows,  we equip $H^{s}_p(\R^d)$ with the norm of $D(A_p^s)$. Moreover,
its fractional power domain space $X_{A_p}^s$ can be characterize by Bessel potential spaces as follows:
\beq\label{Ap}
X_{A_p}^s=H^{s}_p(\R^d)  \quad \forall\, s\in \R
\eeq
(see e.g. \cite[Section 8.3]{Hasse}).

Let  $\phi_0\in C_0^\infty(\R^n)$ be such that $\phi_0(\xi)=1$ for $|\xi|\leq 1$ and $\phi_0(\xi)=0$ for $|\xi|\geq 2$. Moreover, $\phi_j(\xi):=\phi_0(2^{-j}\xi)-\phi_0(2^{-j+1}\xi)$ for $j\in \mathbb{N}$.  Then, we define
$$
S_j:  \mathcal{S}(\R^d)'\to C^\infty(\R^d),\quad
S_j u:= \mathcal{F}^{-1}[\phi_j(\xi)\hat{u}(\xi)] \quad \forall u\in  \mathcal{S}(\R^d)'\quad \forall j\in \mathbb{N}_0.
$$
For $s\in \R$, $1\leq p,q\leq \infty$, the  Besov space $B^s_{p,q}(\R^d)$ is defined by
$$
B^s_{p,q}(\R^d):=\{u\in  \mathcal{S}(\R^d)'\mid \|u\|_{B^s_{p,q}(\R^d)}<\infty\},
$$
where
$$
\|u\|_{B^s_{p,q}(\R^d)}:=\begin{cases}
\displaystyle \left(\sum_{j=0}^{\infty} 2^{jsq}\|S_j u\|^q_{L^p(\R^d)} \right)^{\f{1}{q}}\quad &q<\infty,\\
\displaystyle \sup_{j\geq 0}\{2^{js}\|S_j u\|_{L^p(\R^d)}\}\quad &q=\infty.
\end{cases}
$$
Other definitions of equivalent norms in Besov spaces can be found in  \cite{Triebel}.

\begin{Proposition}\label{the:besov}
Let $a_0\geq 0$, $a_1\in (0,1)$, $1\leq q_1\leq \infty$, and $1<p<\infty$. Suppose that
$$
X=L^p(\Omega),\quad V=B^{a_1}_{p,q_1}(\R^d),
$$
and there exists a function $f\in \E (S_\varphi)$ for some $0<\varphi<\pi$ such that the mappings  $z\mapsto z^{-(a_0+1)}(f(z)-1)$ and   $z\mapsto z f(z)$ are of class $\E (S_\varphi)$. Then,  Assumption \ref{scale} holds with  $ U:=H^{-a_0}_p(\R^d)$,
 $$
X_s:=B^{sa_1}_{p,q}(\R^d)\quad \forall s\in (0,1],\quad X_0:=X,\quad \textrm{and}\quad a:=\frac{a_0}{a_1},
 $$
 where  $q:=\max\{2,p\}$, and
 $$
 P_t:=f(t^{\f{1}{a_1}} A_p)\quad \forall \,t>0.
 $$
\end{Proposition}

\begin{Remark}

 If $a_0=1$, then the mappings
$$
z\mapsto f(z):=e^{-z^2},\quad z\mapsto z^{-2}(f(z)-1),\quad z\mapsto z f(z)
$$
are of class $\E (S_\varphi)$ for any $0<\varphi<\f{\pi}{4}$ by { Lemma \ref{Lemma:es}}.
  Thus, the function $f$ satisfies all assumptions of Proposition \ref{the:besov}. According to \cite[Proposition 8.3.1.]{Hasse}, it holds that
$$
(P_t u)(x)=\frac{e^{-\f{2}{a_1} t}}{(4\pi t^{\f{2}{a_1}} )^{-\f{d}{2}}}\int_{\R^d} e^{-\f{|x-y|^2}{ t^{2/a_1}}}u(y) dy \quad \forall u\in L^p(\Omega),
$$
 which  gives an explicit expression of decomposition operators.

\end{Remark}

\begin{proof}
Due to the well-known results \cite[{Subsection 2.8.1. Remark 1 \& 2}]{Triebel},  the embeddings   $B^{a_1}_{p,q_1}(\R^d)\embed L^p(\R^d)\embed H^{-a_0}_p(\R^d)$, and
$B^{s_1}_{p,q}(\R^d)\embed B^{s_2}_{p,q}(\R^d)$ for  any $s_1>s_2$ are continuous. Thus, both (i) and
(ii) of Assumption \ref{scale} are valid.

On the other hand,  according to Theorem 2 in \cite[Section 2.4.2]{Triebel}, for any $-\infty<s_0,s_1<\infty$, $s_0\neq s_1$, $1<p<\infty$,  $1\leq q_0, q\leq \infty$ and $\Theta\in (0,1)$,  it holds for $s=(1-\Theta)s_0+\Theta s_1$ that
\beq\label{besov0}
( B^{s_0}_{p,q_0}(\R^d), H^{s_1}_p(\R^d))_{\Theta, q}=( H^{s_0}_{p}(\R^d), H^{s_1}_p(\R^d))_{\Theta, q}=B^{ s}_{p,q}(\R^d)
\eeq
with equivalent norms.  In particular, \eqref{besov0} ensures that, for all $s\in (0,1]$ and $r\in [0,s)$,
\beq
( B^{a_1s}_{p,q}(\R^d), H^{-a_0}_p(\R^d))_{\f{s-r}{ s+a}, q}=B^{a_1 r}_{p,q}(\R^d),
\eeq
which according to  \eqref{compreal} yields
\beq\label{besov1}
B^{a_1 r}_{p,q}(\R^d) \in J_{\frac{s-r}{s+a}} ( B^{a_1s}_{p,q}(\R^d),H^{-a_0}_p(\R^d) ).
\eeq
That is,  $X_r\in   J_{\frac{s-r}{s+a}} (X_s,U)$ for all
$s\in (0,1]$ and $r\in (0,s)$, and  $B^{0}_{p,q}(\R^d)\in   J_{\frac{s}{a+s}} (X_s,U)$
for $s\in (0,1]$. Then, from the continuous embedding
$
B^{0}_{p,q}(\R^d)\embed X
$
it follows that $X \in   J_{\frac{s}{a+s}} (X_s,U)$
for $s\in (0,1]$. In conclusion, Assumption  \ref{scale} (iii) is valid.

By the arguments used in \eqref{sect:decomp1}-\eqref{sect:decomp2}
and using { \eqref{Ap} as well as } the facts that $z\mapsto f(z)$, $z\mapsto zf(z)$, $ z\mapsto z^{-a_0-1} (f(z)-1)$ are of class $\E(S_\varphi)$, we can infer that there exists a constant $C_{ap}>0$ such that
\begin{align}
 \|(f(t^{\f{1}{a_1}} A_p)-\textrm{id})\|_{U\to U}=\|(f(t^{\f{1}{a_1}} A_p)-\textrm{id})\|_{X\to X}&\leq C_{ap} \quad\forall\, t>0\label{besov2}, \\
{\|(f(t^{\f{1}{a_1}} A_p)-\textrm{id})x\|_{D(A_p)\to U}=\|A_p^{a_0+1}(f(t^{\f{1}{a_1}} A_p)-\textrm{id})x\|_{X\to X}}
&\leq C_{ap} t^{\f{a_0+1}{a_1}}  \quad\forall\, t>0 \label{besov3} ,\\
\|f(t^{\f{1}{a_1}} A_p)\|_{X\to D(A_p)}= \|A_p f(t^{\f{1}{a_1}} A_p)\|_{X\to X}&\leq C_{ap} t^{-1/a_1} \, \forall\, t>0   \label{besov4},\\
\|f(t^{\f{1}{a_1}} A_p)\|_{X\to X}=\|f(t^{\f{1}{a_1}} A_p)\|_{D(A_p)\to D(A_p)}&\leq C_{ap}  \quad\forall\, t>0   \label{besov5}.
\end{align}
Due to \eqref{besov0}, the interpolation result
\beq\label{inter:Xs}
X_s=(D(A_p),U)_{\f{1-a_1s}{1+a_0},q}=(D(A_p),X)_{1-s,q}
\eeq
holds with equivalent norms.  Selecting $\mathcal{X}_1=\mathcal{X}_2=D(A_p)$, $\mathcal{Y}_1=\mathcal{Y}_1=X$,
$\tau=1-s$  in
Lemma \ref{lemma:interp-real}, we have  by \eqref{inter:Xs}-\eqref{besov5} that for all $s\in (0,1)$
\beq
\|f(t^{\f{1}{a_1}} A_p)\|_{X_s\to X_s}
\leq C \|f(t^{\f{1}{a_1}} A_p)\|_{(D(A_p),X)_{1-s,q} \to (D(A_p),X)_{1-s,q} }
\leq C C_{ap}.
\eeq
Here and afterforward  $C>0$ denote a genetic constant associated with the embedding between equivalent Banach spaces. Hence,  Assumption \ref{scale} (iv) holds.

Then, by taking $\mathcal{Y}_1=\mathcal{Y}_2=U$, $\mathcal X_1=D(A_p)$, $\mathcal{X}_2=U$,  $\tau=\f{1-a_1s}{1+a_0}$
 in
Lemma \ref{lemma:interp-real}  and recalling that $(U,U)_{\f{1-a_1s}{1+a_0},q}=U$ holds with equivalent norms,  we conclude that
\begin{align}\label{besov7}
\|f(t^{\f{1}{a_1}} A_p)-\textrm{id}\|_{X_s\to U}\leq C\|&f(t^{\f{1}{a_1}} A_p)-\textrm{id}\|^{\f{a_1s+a_0}{1+a_0}}_{D(A_p) \to U} \|f(t^{\f{1}{a_1}} A_p)-\textrm{id}\|_{U \to U}^{\f{1-a_1s}{1+a_0}}\notag\\
\underbrace{\leq}_{ \eqref{besov2}-\eqref{besov3}}&  C  C_{ap}t^{s+a},
\end{align}
which proves the first inequality in  (v) of Assumption \ref{scale}.
Similarly, in view of the interpolation result
\beq\label{besov8}
(D(A_p),X)_{1-a_1,q_1}=V \,\, \text{with equivalent norm}
\eeq
due to the second identity of \eqref{besov0} and  \eqref{besov4}-\eqref{besov5}, we  obtain by Lemma \ref{lemma:interp-real} that
\beq\label{besov9}
\|f(t^{\f{1}{a_1}} A_p)\|_{X\to V}\leq C \|f(t^{\f{1}{a_1}} A_p)\|_{X\to X}^{1-a_1 }\|f(t^{\f{1}{a_1}} A_p)\|_{X\to D(A_p)}^{a_1 }
\leq C C_{ap} t^{-1}.
\eeq
On the other hand,  Lemma \ref{lemma:interp-real} in combination with  \eqref{besov5} and the interpolation result \eqref{besov8}  implies
\beq\label{besov10}
\|f(t^{\f{1}{a_1}} A_p)\|_{V\to V}\leq  C\|f(t^{\f{1}{a_1}} A_p)\|_{D(A_p)\to D(A_p)}^{1-a_1 }\|f(t^{\f{1}{a_1}} A_p)\|_{X\to X}^{a_1 }
\leq CC_{ap}
\eeq
for $V=(X,D(A_p))_{a_1,q_1}$.
Choosing $\mathcal X_1=X$, $\mathcal X_2=V$ and $\mathcal Y_1=\mathcal Y_2=V$ in Lemma \ref{lemma:interp-real}, we finally obtain
\beq\label{besov11}
\|f(t^{\f{1}{a_1}} A_p)\|_{X_s\to V}\leq C\|f(t^{\f{1}{a_1}} A_p)\|_{X \to V}^{1-s }\|f(t^{\f{1}{a_1}} A_p)\|_{V \to V}^{s}
\!\!\!\!\!\underbrace{\leq}_{  \eqref{besov9}-\eqref{besov10}}\!\!\!\!\!\!\! CC_{ap}t^{s-1}.
\eeq
This proves the second inequality in Assumption \ref{scale} (v).

\end{proof}

}

\subsection{Inverse radiative problem}\label{sec:app}

	In this section, we  apply  our theory to   an inverse radiative problem. Although convergence rates have been well investigated for the Tikhonov regularization method in the inverse elliptic or parabolic radiativity problems (see e.g.\cite{chen2020,engl89,engl2000,dinh10,bangti11,jiang12}), all these convergence results are established under the assumption that unknown true solution has a finite  penalty value. In the following, we   focus on the case that the unknown radiativity  fails to have a finite penalty value.

As a preparation, let us recall  the notion of  the Bessel potential space on a bounded domain.
 For $p\in (1,\infty)$ and a bounded domain $ U \subset \R^d$ with
a Lipschitz boundary $\p U$,  the space  $H^s_p(U)$ with a possibly non-integer exponent $s\geq 0$ is defined as the space
of all complex-valued functions $v\in L^p(U)$ satisfying
$\hat v_{\vert U}=v$ for some $\hat v \in H^s_p(\R^d)$, endowed with the norm
\beq\label{def:sobolev}
\|v\|_{H^s_p(U)}:=\inf_{\substack{\hat v_{\vert U}=v \\ \hat v \in H^s_p(\R^d)}}\|\hat v\|_{H^s_p(\R^d)}.
\eeq
If $s$ is a non-negative integer, then  $H^s_p(U)$  coincides with the classical Sobolev space $W^{s,p}(U)$. The space  $\mathring{H}^s_p(U)$ denotes the closure of $
C_0^\infty(U)$ in $H^s_p(U)$, and the dual of $\mathring{H}^s_p(U)$ is denoted by $H^{-s}_{q}(U)$ with $q=\f{p}{p-1}$.

\begin{Lemma}\label{lemma:mul}
	Let $1<p<\infty$. Then,   it holds that
	$$
	\|fg\|_{H^{-1}_{p}(\Omega)}\leq \|f\|_{W^{1,\infty}(\Omega)}\|g\|_{H^{-1}_p(\Omega)} \quad \forall\, (f,g)\in W^{1,\infty}(\Omega)\times L^p(\Omega).
	$$

\end{Lemma}
\begin{proof}
Let $f\in  W^{1,\infty}(\Omega)$,  $g\in  L^p(\Omega)$, and $q=\frac{p}{p-1}$.
For any $v\in \mathring{H}_q^1(\Omega)$, an interplay of Libniz's rule (see e.g. \cite{Yagi}) and H\"{o}lder's inequality  yields
	$$
	\|fv\|_{\mathring{H}_q^1(\Omega)}\leq \|f\|_{W^{1,\infty}(\Omega)}\|v\|_{\mathring{H}^1_q(\Omega)}.
	$$
	This inequality together with the fact that
	$$
	\|fg\|_{H^{-1}_p(\Omega)}=\sup_{\|v\|_{\mathring{H}^1_q(\Omega)}=1}|\int_{\Omega} f g v dx|
	$$
	implies the desired result.
	
\end{proof}
Let us consider the following elliptic equation:
\begin{equation}
\left\{ \begin{array}{rlllc}
-\nabla \cdot(a \nabla u )+(\chi_0+\chi )u &=& f  &\text{in}& \Omega,\\
u&=&g &\text{on}&\partial\Omega.
\end{array}
\right. \label{q1}
\end{equation}
In this setting,  $\Omega\subset \mathbb{R}^d$ ($d\geq 2$) is a bounded domain with a $C^2$-boundary $\p\Omega$,  $a\in C^1(\overline{\Omega})$ satisfying  $\min\limits_{x\in \overline{\Omega}}a(x)>0$,
$g\in W^{2-1/p,p}(\p\Omega)$ (see \cite{gri85}) and $f \in L^p(\Omega)$ with $d<p<\infty$.
Moreover, $\chi_0\in L^\infty(\Omega)$ is a non-negative function.  We are interested in recovering the unknown radiativity $\chi$ in the  following admissible set:
$$
K_R:=\{\chi \in L^p(\Omega)\mid 0\leq \chi(x)\leq R\,\,\text{for}\,\text{a.e.}\, x\in \Omega\},\quad R\in (0,\infty)
$$
 from the noisy data $ u^\delta\in H^1_p(\Omega)$ of  the true solution $u^\dag$  satisfying
\beq\label{noisy:pde}
\|u^\delta-u^\dag\|_{ H^1_p(\Omega)}\leq \delta,
\eeq
where $\delta>0$ represents the noisy level. In the following, we summarize the well-posed result
for the elliptic equation \eqref{q1}.

\begin{Lemma}[see\cite{elschner2007optimal}]\label{pro:elliptic}
	 For every $\chi\in K_R$, the elliptic equation \eqref{q1} admits a unique strong solution $u(\chi)\in H^2_p(\Omega)$ satisfying
	\beq\label{ukb}
	C_{E}:=\sup_{\chi\in K_R}\|u(\chi)\|_{H^2_p(\Omega)}<\infty.
	\eeq
	Moreover, the operator $B_\chi:\mathring{H}^1_p(\Omega)\to H^{-1}_p(\Omega)$, defined by
	\beq\label{eq:Bp}
	\langle B_\chi u,v\rangle_{ H^{-1}_p(\Omega),\mathring H^1_q(\Omega)}:=\int_\Omega a \nabla u\cdot \nabla v+(\chi_0+\chi) uv dx \quad \forall (u,v)\in \mathring H^1_p(\Omega)\times \mathring H^1_q(\Omega)
	\eeq
	is a topological isomorphism satisfying
	\beq
	C_{B}:=\sup_{\chi\in K_R}\max\{\|B_\chi\|_{ \mathring{H}^1_p(\Omega)\to H^{-1}_p(\Omega) },\|B_\chi^{-1}\|_{H^{-1}_p(\Omega)\to \mathring{H}^1_p(\Omega)}\}<\infty.
	\eeq

\end{Lemma}

In view of Lemma \ref{pro:elliptic}, if we set $X=L^p(\Omega)$,  $Y=H^1_p(\Omega)$,  and $F:   D(F)\to Y$ by
\beq\label{FD}
F(\chi):=u(\chi) \quad \forall \chi\in D(F):=K_R,
\eeq
 then the inverse radiativity problem is equivalent to the operator equation \eqref{eq:opeq}.   The associated Tikhonov regularization method       reads as follows:
\begin{equation}\label{Tik:pde}
\left\{
\begin{aligned}
&\textrm{Minimize} \quad \|u(\chi)-u^\delta\|_{ H^1_p(\Omega)}^p+\kappa\|\chi\|^p_{\mathring{H}^1_p(\Omega)},\\
& \textrm{subject to} \quad    \chi \in K_R.
\end{aligned}
\right .
\end{equation}
 We underline that  \eqref{Tik:pde} is  oversmoothing since in general the true solution  $\chi^\dag$ does not admit  the regularity property in  $\mathring{H}^1_p(\Omega)$, i.e., $\|\chi\|_{\mathring{H}^1_p(\Omega)}=\infty $.  Let us demonstrate that  Theorem \ref{main:sectorial} applies to  the oversmoothing Tikhonov problem
\eqref{Tik:pde}. To this end, we verify that all assumptions of Theorem \ref{main:sectorial} are satisfied. First of all, the existence of an invertible sectorial operator with efficient domain identical to  $\mathring{H}^1_p(\Omega)$ is guaranteed by the following lemma:

\begin{Lemma}[see{\cite{ChenYousept20}}]\label{ex:PL}
	Let   $A_p: D(A_p)\subset L^p(\Omega)\to L^p(\Omega)$ be given by
			$A_p u:=-\Delta u $  and $D(A):=\{ u\in H_p^2(\Omega)\mid \gamma u=0\}$. Then $A_p: D(A_p)\subset L^p(\Omega)\to L^p(\Omega)$ is a 0-sectorial operator such that its fractional power spaces $D(A_{p}^\theta)$ can be characterized as follows:
			\beq\label{space:Ap}
			D(A^\theta_{p})=\begin{cases} H^{2\theta}_p(\Omega)\, &0\leq\theta<\f{1}{2p},\\
				\{H^{2\theta}_p(\Omega)\mid \gamma u=0\}\, &  1\geq \theta>\f{1}{2p}\,\text{and}\,\theta\neq  \f{p+1}{2p}.
			\end{cases}
			\eeq
\end{Lemma}
Setting  $A:=A_p^{\f{1}{2}}$,    Lemma \ref{ex:PL} and  Lemma \ref{lemma:interp} (iii) imply that  $A:D(A)\subset  L^p(\Omega)\to L^p(\Omega)$ is an invertible $0$-sectorial operator  with		$
		D(A)=\mathring{H}^1_p(\Omega)
		$
		with norm equivalence  $\|\cdot\|_{\mathring{H}^1_p(\Omega)} \sim \|A \cdot\|_{L^p(\Omega)}$. In particular, \eqref{eq:V} is satisfied with $V=\mathring{H}^1_p(\Omega)$.  Moreover,  the adjoint operator of it is exactly $A_{q}^{1/2}:D(A_q^{1/2})\subset L^q(\Omega)\to L^q(\Omega)$ with $q=\f{p}{p-1}$. Then, it follows from \eqref{frac:dual} and \eqref{space:Ap} that
		\beq\label{A-1}
		\|A^{-1} x\|_{L^p(\Omega)}\sim \|x\|_{H^{-1}_p(\Omega)}\quad \forall x\in L^p(\Omega).
		\eeq
 In view of \eqref{A-1},  the upcoming  lemma shows that Assumption 1 holds with  the forward operator $F$ given by \eqref{FD}, $a=1$, and $U=H_p^{-1}(\Omega)$. 		

		\begin{Lemma}\label{the:stab}
	If    $|u(\chi^\dag)|>c_0$ holds in $\overline{\Omega}$ for some positive constant $c_0>0$, then there exists a constant $C>0$ such that
	$$
	\f 1 C \|\chi-\chi^\dag\|_{H^{-1}_p(\Omega)} \leq \|u(\chi)-u(\chi^\dag)\|_{ {H}_p^1(\Omega)}\leq C\|\chi-\chi^\dag\|_{H^{-1}_p(\Omega)}\quad \forall\,\chi\in K_R.
	$$
	
\end{Lemma}
\begin{proof}
 Let $q=\f{p}{p-1}$.
	By the definition \eqref{q1}, for each $\chi\in K_R$,  the function $w:=u(\chi)-u(\chi^\dag)$ satisfies
	$$
	\int_\Omega a \nabla w\cdot \nabla v+(\chi_0+\chi) w v dx=\int_\Omega
	(\chi^\dag-\chi)u(\chi^\dag) v dx\quad \forall\, v\in \mathring{H}_q^1(\Omega),
	$$
	which is equivalent to
	$B_\chi w=(\chi^\dag-\chi)u(\chi^\dag)$ where $B_\chi:\mathring{H}^1_p(\Omega)\to H^{-1}_p(\Omega)$ is defined as in  \eqref{eq:Bp}. Then, Lemma \ref{pro:elliptic} ensures that
	\beq\label{iso1}
	\f{1}{C_B}\| u(\chi^\dag)(\chi-\chi^\dag)\|_{H_p^{-1}(\Omega)}
	\leq \|u(\chi)-u(\chi^\dag)\|_{ {H}_p^1(\Omega)} \leq C_B\| u(\chi^\dag)(\chi-\chi^\dag)\|_{H_p^{-1}(\Omega)}.
	\eeq
	On the other hand, 	Lemmas \ref{lemma:mul} and \ref{pro:elliptic}  ensure that
	\begin{align}\label{iso2}
	\| u(\chi^\dag)(\chi-\chi^\dag)\|_{H_p^{-1}(\Omega)}\leq& \|u(\chi^\dag)\|_{W^{1,\infty}(\Omega)}\|\chi-\chi^\dag\|_{H_p^{-1}(\Omega)}\notag\\
	\leq& C_m\|u(\chi^\dag)\|_{H^2_p(\Omega)} \|\chi-\chi^\dag\|_{H_p^{-1}(\Omega)},
	\end{align}
	where $C_m>0$ denotes the embedding constant associated with  $H^{2}_p(\Omega) \embed W^{1,\infty}(\Omega)$ as $p > d$. On the other hand, using the Leibniz rule and the condition $|u(\chi^\dag)|>c_0$, we obtain that $u(\chi^\dag)^{-1}\in  H^{2}_p(\Omega)$. Then, again by Lemma \ref{lemma:mul},
		\beq\label{iso3}
	\|\chi-\chi^\dag\|_{H_p^{-1}(\Omega)}
	\leq C_m\|u(\chi^\dag)^{-1}\|_{H_p^2(\Omega)}\|u(\chi^\dag)(\chi-\chi^\dag)\|_{H_p^{-1}(\Omega)}.
	\eeq
	The combination of \eqref{iso1}-\eqref{iso3} yields the desired result.
	
\end{proof}

Now, it remains to verify     the last assumption of   Theorem \ref{main:sectorial} concerning the existence of the holomorphic function $f$ with properties as in Theorem \ref{main:sectorial}.  Let us set $f(z):=e^{-z^2}$.  This function obviously belongs to $\E(S_\varphi)$ for any  $0<\varphi<\f{\pi}{4}$.  In the following,  let $\varphi\in (0,\pi/4)$ be arbitrarily  fixed. Since $e^{-z^2}=1-z^2+o(z^4)$ as $z\to 0$, and  the mapping $z\mapsto ze^{-z^2}$ belongs to $\E(S_\varphi)$,  it follows by Lemma \ref{Lemma:es} that for any $s\in (0,1)$,  the mappings $z\mapsto z^{-(1+s)} (f(z)-1)$ and $z\mapsto z^s f(z)$ are of class $\E(S_\varphi)$.  Let us now verify the final condition \eqref{sect:invariance} in    Theorem \ref{main:sectorial}. First, as $A:=A_p^{\f{1}{2}}$,  we have that
		\beq\label{Pl}
		f(tA) =e^{-(tA)^2}=e^{-t^2A_p}\quad t>0.
		\eeq
	Furthermore,  according to \cite[Corollary 4.3 and Theorem 4.9]{{ouhabaz2009analysis}}, it holds that
		 $$
		 \|e^{-A_p t} x\|_{L^\infty(\Omega)}\leq
		\|x\|_{L^\infty(\Omega)}\quad \forall\,(x,t)\in L^\infty(\Omega)\times (0,\infty)
		 $$
		 and
		 $$
		 e^{-t A_p} x\geq 0\,\, \textup{a.e. in}\,\,\Omega \quad \text{for all non-negative} \, x\in L^p(\Omega) \,\,\text{and all}\,\,t\geq 0.
		 $$
As a consequence,
		$$
		f(tA) K_R\subset K_R\quad \forall t\geq 0,
		$$
		which yields   the desired condition \eqref{sect:invariance} for    Theorem \ref{main:sectorial}.  Altogether,  we have verified all requirements of  Theorem \ref{main:sectorial} for the oversmoothing Tikhonov problem
\eqref{Tik:pde}, leading to the following result:

\begin{Corollary}

	Suppose  that $\chi^\dag\in K_R$  and  $|u(\chi^\dag)|>c_0$ for some positive constant $c_0$. If $\chi^\dag\in  {H}_p^\theta(\Omega)$ with $\theta\in (0,\f{1}{p})$ or
	$\chi^\dag\in \mathring{{H}}_p^\theta(\Omega)$ for some $\theta\in (\f{1}{p},1)$, then every minimizer $\chi^\delta_\kappa$ of  \eqref{Tik:pde} satisfies
	$$
	\|\chi^\delta_\kappa-\chi^\dag\|_{L^p(\Omega)}=O(\delta^{\f{\theta}{1+\theta}})\quad \textup{as}\,\,\delta \to 0.
	$$
\end{Corollary}
\begin{Remark}
The application of the developed theory (Theorems \ref{main:sectorial} and \ref{main}) can be extended to more complicated nonlinear PDEs with low regularity arising in particular from applications in nonlinear electromagnetic inverse and optimal design problems \cite{chenyousept2018,LamYousept2020,Yousept2012,Yousept2013,Yousept2017,Yousept2020a,Yousept2020b}.
\end{Remark}

\section{Appendix: real and complex interpolation} \label{sec:appen}

We follow \cite[Appendix B]{Hasse} for the definition of the real and complex interpolation spaces. Let $(X,Y)$ be an interpolation couple. For given $x\in X+Y$ and $t>0$,  we introduce
$$
K(t,x):=K(t,x,X,Y):=\inf\{\|a\|_X+t\|b\|_Y\mid x=a+b,a\in X,b\in Y\}.
$$
Moreover, for given $p\in [1,\infty]$,  let $L_\star^p(0,\infty)$ denote the space
of $p$-integrable functions on $(0,+\infty)$ with respect to the measure $\frac{1}{t}dt$.
 Now, we introduce for $\theta\in (0,1)$ and $p\in [1,\infty]$, the real interpolation space  as follows:
$$
(X,Y)_{\theta,p}:=\{x\in X+Y\mid  \text{the function} \, (0,\infty)\ni t\mapsto t^{-\theta}K(t,x)\in \R \,\,\textup{belongs to}\,\, L_\star^p(0,\infty) \},
$$
endowed with the norm
$$
\|x\|_{(X,Y)_{\theta,p}}:=\begin{cases}
\displaystyle \left(\int_0^\infty t^{-p\theta-1}K(t,x)^p dt\right)^{1/p}\quad &p\in[1,\infty),\\
\sup_{t\in (0,\infty)}\{t^{-\theta}K(t,x)\}\quad &p=\infty.
\end{cases}
$$
In the following, let  $S$  denote the vertical strip, i.e.,
$$
S:=\{z\in \mathbb C\mid 0<\Re z<1\},
$$
where $\Re z$ denotes the real part of $z\in \mathbb C$.
\begin{Definition}
Let $(X,Y)$ be an interpolation couple. The space $\mathcal{F}(X,Y)$ consists of all functions $f:S\to X+Y$ satisfying the following properties:
\begin{enumerate}

\item  $f$ is holomorphic in the interior of the strip,   continuous and bounded up to its boundary
with values in $X+Y$.

\item The function
$$
f_0(t):=f(it)\quad \quad \forall\, t\in\R
$$  is bounded and continuous with values in $X$.

\item The function
$$
f_1(t):=f(1+it)\quad \quad \forall\, t\in\R
$$
is bounded and continuous with values in $Y$.

\end{enumerate}
Moreover, the space  $\mathcal{F}(X,Y)$  is endowed
with the   norm
$$
\|f\|_{\mathcal{F}(X,Y)}=\max\{\sup_{t\in \R}\|f_0(t)\|,\sup_{t\in \R}\|f_1(t)\| \}.
$$
\end{Definition}
For every $\theta\in [0,1]$, we define the complex interpolation  space by
$$
[X,Y]_\theta:=\{f(\theta) \in X + Y \mid f\in \mathcal{F}(X,Y)\},
$$
equipped with the norm
$$
\|x\|_{[X,Y]_\theta}=\inf_{f\in  \mathcal{F}(X,Y) \& f(\theta)=x }\|f\|_{\mathcal{F}(X,Y)}.
$$

 \end{document}